\documentclass[a4paper, 11 pt, twoside]{amsart}
\usepackage{srollens-en}[2011/03/03]
\usepackage[left = 3.5cm, right = 3.5cm, headsep = 6mm,
footskip = 10mm, top = 35mm, bottom = 35mm, footnotesep=5mm, headheight =
2cm]{geometry}

\usepackage[T1]{fontenc}

\usepackage{tikz-cd}

\usepackage[hypertexnames=false,
backref=page,
	pdfpagemode=UseNone,
	breaklinks=true,
	extension=pdf,
	colorlinks=true, 
	linkcolor=blue,
	citecolor=blue,
	urlcolor=blue,
]{hyperref}

{\theoremstyle{Lehn-it}
\newtheorem{propdef}[theo]{Proposition/Definition}
}
{\theoremstyle{Lehn-up}
\newtheorem{notation}[theo]{Notation}
}

\renewcommand{\restr}[1]{{\raisebox{-0.0\height}{$\mid_{#1}$}}}
\renewcommand{\epsilon}{\varepsilon}
\renewcommand{\phi}{\varphi}

\newcommand{\Diff}{\mathrm{Diff}}
\newcommand{\sing}[1]{{{#1}_{\mathrm{sing}}}}

\newcommand{\red}{\mathrm{red}}
\DeclareMathOperator{\supp}{\mathrm{supp}}

\renewcommand{\tilde}{\widetilde}

\title{Pluricanonical maps of stable log surfaces}
\author{Wenfei Liu}
\address{Wenfei Liu \\Institut f\"ur algebraische Geometrie\\  Gottfried Wilhelm Leibniz Universit\"at Hannover\\Welfengarten 1\\ 30167 Hannover\\Germany}
\email{wliu@math.uni-hannover.de}

\author{S\"onke Rollenske}
\address{S\"onke Rollenske\\Fakult\"at f\"ur Mathematik\\Universt\"at Bielefeld\\Universit\"atsstr. 25\\33615 Bielefeld\\Germany}
\email{rollenske@math.uni-bielefeld.de}

\begin{document}
\begin{abstract}
Stable surfaces and their log analogues are the type of varieties naturally occurring as boundary points in moduli spaces. We extend classical results of Kodaira and Bombieri to this more general setting: if $(X,\Delta)$ is a stable log surface with  reduced boundary (possibly empty) and $I$ is its global index,  then $4I(K_X+\Delta)$ is base-point-free and $8I(K_X+\Delta)$ is very ample.

These bounds can be improved under further assumptions on the singularities or invariants, for example, $5(K_X+\Delta)$ is very ample if $(X,\Delta)$ has semi-canonical singularities.
\end{abstract}

\maketitle
\setcounter{tocdepth}{1}
\tableofcontents

\section{Introduction}
It is a general fact that moduli spaces of \emph{nice} objects in algebraic geometry, say smooth varieties,  are often non-compact. But usually there is a modular compactification where the boundary points correspond to related, but more complicated objects.

Such a modular compactification has been known for the moduli space $\km_g$ of smooth curves of genus $g$ for a long time and in \cite{ksb88} Koll\'ar and Shepherd-Barron made the first step towards the construction of a modular compactification $\overline{\gothM}$ for the moduli space $\mathfrak M$ of surfaces of general type. Even though the actual construction of the moduli space was delayed for several decades because of formidable technical obstacles to be overcome, it was clear from the beginning that the objects parametrised by $\overline{\gothM}$ should be surfaces with semi-log-canonical singularities and ample canonical divisor, for short \emph{stable surfaces}.

A more general version also incorporates the possibility of a (reduced) boundary divisor (see Section~\ref{sect: prelims} for the precise definitions); this is the higher dimensional analogue of pointed stable curves and was worked out by Alexeev \cite{alexeev96a, alexeev06}.

In recent years, several components of the moduli space of stable varieties or pairs have been investigated in detail\footnote{The following is a probably incomplete list of results in this direction: \cite{hassett99, lee00, hacking04,vop05,vop06a, vop06b, ale-pard09, HKT09, rollenske10b, liu12, laza12, bhps12, patakfalvi12}.}. On the other hand, many of the standard tools to study, for example, smooth surfaces of general type are not yet available for stable surfaces. In this paper, we make the first steps in the understanding of  pluricanonical maps of such surfaces.

Pluricanonical maps are one of the main tools in the study of smooth surfaces of general type and their canonical models. They have been an active subject of research ever since  Bombieri's seminal paper \cite{bombieri73}.
Recall that the $m$-canonical map of a variety $X$ is the rational map $\phi_m\colon X \dashrightarrow \IP^N$ associated to the linear system $|mK_X|$. Then the roughest version of Bombieri's results says that on  a surface with canonical singularities and ample canonical divisor $\phi_m$ is an embedding for $m\geq5$; it had been proved earlier by Kodaira that $\phi_m$ is a morphism for $m\geq 4$ \cite{kodaira68}. These results are sharp but can be much refined and we refer to \cite[Sect.~VII]{BHPV} or the recent survey \cite{BCP06} for more information.

The singularities of a stable surface can be  much worse than canonical singularities: in general they are non-normal, not Gorenstein and not (semi-)rational. Thus many of the techniques which one can use to prove Bombieri-type theorems do not carry over directly. The following theorem is proved in Section~\ref{section: bpf} by applying  a Reider-type result due to  Kawachi on the normalisation combined with a detailed analysis of the non-normal locus.
\begin{custom}[Theorem~\ref{thm: base-point-free}]
 Let $(X,\Delta)$ be a connected stable log surface with global index $I$.\footnote{Multiplication with the index is clearly necessary since $\phi_m$ cannot be a morphism if $m(K_X+\Delta)$ is not a Cartier divisor. However, it does make sense to ask which is the first pluri-log-canonical map to be birational. In the normal case Langer has given an explicit but still unrealistically large bound \cite[Sect.~5]{langer01b}.}
\begin{enumerate}
\item The line bundle $\omega_X(\Delta)^{[mI]}$  is base-point-free for $m\geq 4$.
\item The line bundle $\omega_X(\Delta)^{[mI]}$  is base-point-free for $m\geq 3$  if one of the following holds:
\begin{enumerate}
\item $I\geq 2$.

\item There is no irreducible component $\bar X_i$ of the normalisation such that $\left(\pi^*(K_X+\Delta)\restr {\bar X_i}\right)^2=1$, and the union of $\Delta$ and the non-normal locus is a nodal curve.
\item $X$ is normal and we do not have $I=(K_X+\Delta)^2=1$.
\end{enumerate}
\end{enumerate}
\end{custom}
\noindent For normal stable surfaces without boundary this recovers \cite[Cor.~3, Cor.~4]{kawachi-masek98}. 

Our results on pluri-log-canonical embeddings are somewhat more involved. We follow an approach  due to Catanese and Franciosi \cite{catanese-franciosi96}, later refined in collaboration with Hulek and Reid \cite{CFHR}: for every subscheme of length two find a pluri-log-canonical curve containing it and then prove that this curve is embedded by $|mI(K_X+\Delta)|$. Without further assumptions on singularities and invariants we get:
\begin{custom}[General bounds (Theorem~\ref{thm: main general})]
Let $(X, \Delta)$ be a connected stable log surface of global index $I$.
\begin{enumerate}
\item The line bundle $\omega_X(\Delta)^{[mI]}$  is very ample for $m\geq 8$.
\item The line bundle $\omega_X(\Delta)^{[mI]}$  defines a birational morphism for $m\geq6$.
\item The line bundle $\omega_X(\Delta)^{[mI]}$  is very ample for $m\geq 6$ if $I\geq2$.
\end{enumerate}
\end{custom}
We do not believe all of these bounds to be sharp. The main obstacles in our proof are the extra contributions from the worse than canonical singularities and the fact that curves containing irreducible components of the non-normal locus and the boundary do not behave well under normalisation. We explain this more in detail in Section~\ref{sect: prelim for embedding}, see also Remark~\ref{rem: not sharp}.

Under additional assumptions we can improve the bounds obtained above. In particular, if $X$ is semi-canonical then we obtain the same bound as in the classical case.
\begin{custom}[Bounds for milder singularities (Theorem~\ref{thm: main special})]
 Let $(X, \Delta)$ be a connected stable log surface of global index $I$ and let $D$ be the non-normal locus of $X$.
\begin{enumerate}
 \item The line bundle $\omega_X(\Delta)^{[mI]}$  is very ample  for $m\geq 7$ if one of the following holds:
\begin{enumerate}
\item There is no irreducible component $\bar X_i$ of the normalisation such that $\left(\pi^*(K_X+\Delta)\restr {\bar X_i}\right)^2=1$, and the union of $\Delta$ and the non-normal locus is a nodal curve.
\item  $X$ is normal and not $(K_X+\Delta)^2=1$.
\end{enumerate}
\item The line bundle $\omega_X(\Delta)^{[mI]}$  is very ample  for $m\geq6$ if the normalisation $\bar X$ is smooth along the conductor divisor and has at most canonical singularities elsewhere. 
 \item The line bundle $\omega_X(\Delta)^{[mI]}$  is very ample  for $m\geq5$ if $D\cup\Delta$ is a nodal curve,  $\bar X$ is smooth along the conductor divisor, and $X\setminus D$ has at most canonical singularities.

In particular these conditions are satisfied, if $(X,\Delta)$ has semi-canonical singularities.
\end{enumerate}
\end{custom}

For a connected stable surface $X$  with canonical singularities the bi-canonical map is a morphism as soon as $K_X^2\geq 5$  and the tri-canonical map is an embedding as soon as $K_X^2\geq 6$ (see \cite{catanese87}).  Such behaviour cannot be expected for stable surfaces: in Example~\ref{exam: large K^2} we construct an irreducible, Gorenstein stable surface with $K_X^2$ arbitrarily large such that  the bi-canonical map not a morphism and neither the tri-canonical nor the $4$-canonical  map is an embedding. 

A natural extension of the aforementioned results is the study of the log-canonical ring. We do not engage in a detailed study but only state the results that follow by standard methods from Theorem \ref{thm: base-point-free}. 
\begin{custom}[Theorem~\ref{thm: ring}]
Let $(X, \Delta)$ be a stable log surface of index $I$. Then the log-canonical ring,
\[R(X, K_X+\Delta) = \bigoplus_{m\in \IZ_{\geq 0}} H^0(X, \omega_X(\Delta)^{[m]}),\]
is generated in degree at most $12I+1$ and in degree at most  $9I+1$  under the same assumptions as in Theorem~\ref{thm: base-point-free}\refenum{ii}.
\end{custom}
All the results should only be regarded as a first step towards a precise understanding of pluri-log-canonical maps and log-canonical rings of stable log surfaces. 

Our method relies on the rough classification of semi-log-canonical singularities and therefore does not generalise to higher dimensions at the moment.

\subsection*{Acknowledgements:} We had the pleasure to discuss parts of this project with  Fabrizio Catanese, Mi\-chael L\"onne, and Markus Zowislok.  We rely heavily on the results in \cite{KollarSMMP} and are grateful to J\'anos Koll\'ar for sending us a preliminary version. S\'andor Kov\'acs gave us a crucial hint on how to prove the vanishing results in Section~\ref{section: vanishing}. Yongnam Lee convinced us to extend our results to the log case. We are grateful to the anonymous referee for carefully reading the whole manuscript: his remarks lead to a more accurate and, hopefully, also more readable presentation.

Both authors were supported by DFG via the second author's Emmy-Noether project and partially via SFB 701. The first author was partially supported also by the Bielefelder Nachwuchsfonds.

\subsection{Notations and conventions}
We work exclusively with schemes of finite type over the complex numbers.
\begin{itemize}
\item The singular locus of a scheme $X$ will be denoted by $\sing X$.
\item A surface is a reduced, projective scheme  of pure dimension 2 but not necessarily irreducible or connected.
\item A curve is a purely 1-dimensional scheme that  is Cohen--Macaulay. A curve is not assumed to be reduced, irreducible or connected. For a point $p\in C$ we denote by $\mu_p(C)$ its multiplicity.
\item For a sheaf $\kf$ on $X$ we denote  by $\kf^{[m]} = \shom_{\ko_X}(\shom_{\ko_X}(\kf^{\tensor m}, \ko_X), \ko_X)$ the reflexive powers. 

\item We switch back and forth between multiples of a canonical divisor, $mK_X$, and reflexive powers of the canonical sheaf $\omega_X^{[m]}=\ko_X(mK_X)$. See Section \ref{sect: divisors} for a discussion of divisors and associated divisorial sheaves.
\end{itemize}
Some further notation on demi-normal schemes or semi-log-canonical pairs will be fixed in Notation \ref{notation}.

\section{Preliminaries}\label{sect: prelims}
In this section we recall some necessary notions as well as constructions that we need throughout the text. Most of these are available in all dimensions, but for our purpose it suffices to focus on the case of surfaces. Our main reference is \cite[Sect.~5.1--5.3]{KollarSMMP}.

\subsection{Stable log surfaces}
Let $X$ be a demi-normal surface, that is,  $X$ satisfies $S_2$ and  at each point of codimension 1, $X$ is either regular or has an ordinary double point.
We denote by  $\pi\colon \bar X \to X$ the normalisation of $X$. The conductor ideal
$ \shom_{\ko_X}(\pi_*\ko_{\bar X}, \ko_X)$
is an ideal sheaf in both $\ko_X$ and $\ko_{\bar X} $ and as such defines subschemes
$D\subset X \text{ and } \bar D\subset \bar X,$
both reduced and of pure codimension 1; we often refer to $D$ as the non-normal locus of $X$. 

Let $\Delta$ be a reduced curve on $X$ whose support does not contain any irreducible component of $D$. Then the strict transform $\bar \Delta$ in the normalisation is well defined.
\begin{defin}\label{defin: slc}
We call a pair $(X, \Delta)$ as above a  \emph{log surface}; $\Delta$ is called the (reduced) boundary.\footnote{In general one can allow rational coefficients in $\Delta$, but we will not use this here.}

A log surface $(X,\Delta)$ is said to have \emph{semi-log-canonical (slc)}  singularities if it satisfies the following conditions: 
\begin{enumerate}
 \item $K_X + \Delta$ is $\IQ$-Cartier, that is, $m(K_X+\Delta)$ is Cartier for some positive integer $m$; the minimal such $m$ is called the (global) index of $(X,\Delta)$.
\item The pair $(\bar X, \bar D+\bar \Delta)$ has log-canonical singularities. 
\end{enumerate}
The pair $(X,\Delta)$ is called stable log surface if in addition $K_X+\Delta$ is ample. A stable surface is a stable log surface with empty boundary; these are the surfaces relevant for the compactification of the Gieseker moduli space.

By abuse of notation we say $(X, \Delta)$ is a Gorenstein stable log surface if  the index is equal to one, i.e., $K_X+\Delta$ is an ample Cartier divisor.
\end{defin}

Let $(X, \Delta)$ be a log surface. Since $X$ has at most double points in codimension 1 the map $\pi\colon \bar D \to D$ on the conductor divisors is generically a double cover and thus  induces a rational involution on $\bar D$. Normalising the conductor loci we get an honest involution $\tau\colon \bar D^\nu\to \bar D^\nu$ such that $D^\nu = \bar D^\nu/\tau$.

\begin{theo}[{\cite[Thm.~5.13]{KollarSMMP}}]\label{thm: triple}
Associating to a log-surface $(X, \Delta)$ the triple $(\bar X, \bar D+\bar \Delta, \tau\colon \bar D^\nu\to \bar D^\nu)$ induces a one-to-one correspondence
 \[
  \left\{ \text{\begin{minipage}{.12\textwidth}
 \begin{center}
         stable log surfaces  $(X, \Delta)$
 \end{center}
         \end{minipage}}
 \right\} \leftrightarrow
 \left\{ (\bar X, \bar D, \tau)\left|\,\text{\begin{minipage}{.37\textwidth}
   $(\bar X, \bar D+\bar \Delta)$ log-canonical pair with 
  $K_{\bar X}+\bar D+\bar \Delta$ ample, \\
   $\tau\colon \bar D^\nu\to \bar D^\nu$  an involution s.th.\
    $\Diff_{\bar D^\nu}(\Delta)$ is $\tau$-invariant.
            \end{minipage}}\right.
 \right\}.
 \]
 \end{theo}
For the definition of the different see Definition \ref{def: different} below. 

\begin{notation}\label{notation}
 In the rest of the article we continue to use the notation above, repeated here as a diagram:
 \begin{equation}
\begin{tikzcd}
    \bar X \dar{\pi}\rar[hookleftarrow] & \bar D\dar{\pi} & \bar D^\nu \lar[swap]{\bar\nu}\dar{/\tau}
    \\
X\rar[hookleftarrow] &D &D^\nu\lar{\nu}.
    \end{tikzcd}
\end{equation}
An important consequence of Theorem \ref{thm: triple} and its proof is that both squares in the diagram are pushouts.
\end{notation}

\subsection{Semi-resolutions}
It is sometimes  useful to resolve stable surfaces as much as possible while keeping the singularities in codimension 1. 
\begin{defin}[\cite{ksb88}, \cite{KollarSMMP}]\label{defin: semi resolution}
A surface $Y$ is called semi-smooth if every point of $Y$ is either smooth or double normal crossing or a pinch point\footnote{A local model for the pinch point in $\IA^3$ is given by the equation $x^2+yz^2=0$.}. 

A  smooth rational curve $E$  on a semi-smooth surface $Y$ which is not contained in the non-normal locus is called a $(-1)$-curve if $E^2=-1$ and $\deg K_Y\restr E \leq 0$.

A morphism of demi-normal surfaces $f\colon Y\rightarrow X$ is called a semi-resolution if the following conditions are satisfied:
\begin{enumerate}
 \item $Y$ is semi-smooth;
 \item there is a semi-smooth open subscheme $U$ of $X$ such that the codimension of $X\setminus U$ is two and $f$ is an isomorphism over $U$;
 \item $f$ maps the singular locus of $Y$ birationally onto the non-normal locus of $X$.
\end{enumerate}
A semi-resolution is called minimal if it does not contract $(-1)$-curves.
\end{defin}

\begin{theo}[{\cite{vS87}, \cite[Thm.~10.54]{KollarSMMP}}]
 Let $X$ be a demi-normal surface. Then $X$ has a unique minimal semi-resolution. 
\end{theo}
The possible configurations of exceptional divisors on the minimal semi-resolution of an slc point will be discussed in Section \ref{sect: numerical cycle}.
Looking at these possibilities it is easy to see how to incorporate a reduced  boundary into the resolution process for an slc pair: if $(X, \Delta)$ is an slc pair and $f\colon Y\to X$ is the minimal semi-resolution of $X$ then blowing up all intersection points of the non-normal locus of $Y$ and the strict transform $\Delta_Y=(\inverse f)_*\Delta$ we get a semi-resolution of $X$ such that the strict transform of the boundary is contained in the normal locus and it is minimal with 
this property. We call this the minimal log-semi-resolution of $(X, \Delta)$. The general case  in all dimensions is treated in \cite[Sect.~10.4]{KollarSMMP}.

\subsection{Divisors and restrictions to curves}\label{sect: divisors}
Let $X$ be a demi-normal surface. In particular, $X$ is  Gorenstein in codimension 1 and $S_2$ and the theory of generalised divisors from \cite{Hartshorne94} applies to $X$. 

A \emph{divisorial sheaf} on $X$ is a reflexive coherent  $\ko_X$-module that is locally free of rank 1 at the generic points of $X$ \cite[Prop.~2.8]{Hartshorne94} and there is a one-to-one correspondence between divisorial subsheaves of the sheaf of total quotient rings which are contained in $\ko_X$ and closed subschemes of codimension 1 without embedded points \cite[Prop.~2.4]{Hartshorne94}. A  divisorial sheaf is called \emph{almost Cartier} if it is invertible in codimension 1.

A \emph{Weil divisor} (resp.~\emph{$\IQ$-Weil divisor}) on $X$ is a finite, formal, $\IZ$-linear (resp.~$\IQ$-linear) combination $D=\sum_i m_i D_i$ of irreducible and reduced subschemes of codimension 1 \cite[Sect.~1.1]{KollarSMMP}. By a ($\IQ$-)divisor we mean a ($\IQ$-)Weil divisor. We call a $\IQ$-Weil divisor reduced if all non-zero coefficients are equal to $1$. 

Arbitrary Weil divisors containing a component of the non-normal locus do not behave well in many respects, so we often need to exclude them. This is encoded in the following definition.
\begin{defin}\label{def: wb}
A $\IQ$-Weil divisor $B$ on a log surface $(X,\Delta)$ is called \emph{well-behaved} (resp.~\emph{log-well-behaved}) if its support does not contain any irreducible component of $D$ (resp.~$D\cup\Delta$).
\end{defin}
For a well-behaved Weil divisor $B$, the corresponding  divisorial sheaf $\ko_X(B)$ is obtained as follows: let $Z$ be the locus where $B$ is not a Cartier divisor and $U = X\setminus Z\overset\iota\into X$. Then $Z$ is of codimension 2 since $B$ does not contain a component of the non-normal locus and $\ko_X(B):=\iota_*\ko_U(B)$ is an almost Cartier divisorial sheaf. If, in addition, $B$ is effective, then the inclusion $\ko_X(-B)=\iota_*\ko_U(-B)\into \iota_*\ko_U=\ko_X$ defines a subscheme structure on $B_\red$.
With this subscheme structure $B$ satisfies $S_1$, i.e., is Cohen--Macaulay, since $\ko_X(-B)$ and $\ko_X$ satisfy $S_2$. We will not distinguish the subscheme and the divisor in the notation.

On the other hand, given an almost Cartier  divisorial sheaf $\ka$ we can find a well-behaved Weil divisor $A$ such that $\ka=\ko(A)$ \cite[Prop.~2.11]{Hartshorne94}. 

Let $\omega_X$ be the dualising sheaf which coincides with the pushforward of the canonical bundle on the Gorenstein locus. Note that $\omega_X$ is almost Cartier, so there is a well-behaved canonical divisor $K_X$,  defined up to linear equivalence, such that $ \ko_X(K_X) = \omega_X$. By a local computation we have $(\pi^*\omega_X)^{[1]} \isom  \omega_Y(\bar D)$.

Restricting divisors and divisorial sheaves to curves requires some extra care, if the divisor is not Cartier.
\begin{defin}\label{def: restriction divisorial sheaf}
Let $B\subset X$ be a curve, that is, a Cohen--Macaulay subscheme of pure codimension 1, and let $A$ be well-behaved divisor.
Then we define
\[ \ko_B(A) = \ko_X(A)\tensor \ko_B/\text{torsion}.\]
\end{defin}
Note that modding out the torsion subsheaf is in general not equal to taking the double dual if the curve is not Gorenstein \cite[Example~4.1.9]{Kass}.

On the set of torsion-free sheaves of rank 1 on a curve $B$ we can also define a multiplication 
\[ \kf[\tensor ]\kf' = \kf\tensor \kf'/\text{torsion}.\]
This product is well-behaved only if one of the sheaves is a line bundle. For example, it may well happen that the restriction map from divisorial sheaves on $X$ to torsion-free sheaves on $B$ is not multiplicative, that is, in general 
\[ \ko_B(mA)\neq \underbrace{\ko_B(A)[\tensor]\dots[\tensor]\ko_B(a)}_{\text{$m$ times}}.\]
As a concrete example one may consider $A$ and $B$ to be a ruling of the cone over a twisted cubic. 
However, the usual short exact sequences suggested by the notation still work.

\begin{lem}\label{lem: div seq}
 Let $X$ be a demi-normal surface and $B$ a well-behaved curve. Let $A$ be a well-behaved divisor.
\begin{enumerate}
 \item There is an exact sequence
\[
 0\rightarrow \ko_X(A-B) \rightarrow \ko_X(A) \rightarrow \ko_B(A) \rightarrow 0.
\]
\item If $B=B_1+B_2$ is a decomposition of $B$ into a sum of (non-empty) subcurves then there is an exact sequence 
\[
 0\rightarrow \ko_{B_1}(A-B_2) \rightarrow \ko_{B}(A) \rightarrow \ko_{B_2}(A) \rightarrow 0.
\]
\end{enumerate}
\end{lem}
\begin{proof} The composition $\phi:\ko_X(A)\rightarrow \ko_X(A)\restr B \rightarrow\ko_B(A)$ is surjective. By the depth lemma $\ker\phi$ is an $S_2$-sheaf. By looking at the open subscheme of $X$ where $\ko_X(A)$ is Cartier, one sees easily that $\ker\phi = \ko_X(A-B)$ which implies \refenum{i}.

For \refenum{ii} consider the  commutative diagram
\[
 \begin{tikzcd}
{}  & 0 \dar\\
0\rar & \ko_ X (A-B)\rar\dar& \ko_ X (A)\rar\arrow[equal]{d}& \ko_B(A)\rar\dar& 0\\
0\rar & \ko_ X (A-B_2)\rar\dar& \ko_X(A)\rar& \ko_{B_2}(A)\rar& 0\\
&\ko_{B_1}(A-B_2)\dar&\\
&0&
 \end{tikzcd}
\]
where the exactness of the column and rows follows from \refenum{i}. The assertion  follows by the Snake Lemma.
\end{proof}

\subsection{Intersection pairing}
\begin{defin}\label{defin: intersection}
Let $X$ be a demi-normal surface.
\begin{enumerate}
 \item We define a $\IQ$-valued intersection pairing for well-behaved Weil divisors in the following way: let  $A$, $B$ be well-behaved Weil divisors on $X$ and $\bar A$, $\bar B$ their strict transforms the normalisation $\bar X$. Then the intersection number is
\[
 AB:= \bar A \bar B
\]
where we use Mumford's intersection pairing for  normal surfaces. (see e.g. \cite{sak84}).
\item For a well-behaved $\IQ$-Cartier Weil divisor $F$ and a curve $B$ on $X$ we denote by 
 \[
 \deg{F\restr{B}} := \frac{1}{m}\deg{\ko_X(mF)\restr{B}},                                                                      
\]
where $m$ is a positive integer such that $mF$ is Cartier. This could be called the numerical degree of $F$ on $B$. Note that $\deg F\restr B = FB$ always holds.
\item For a torsion-free sheaf $\kf$ on a  Cohen--Macaulay curve $B$ we define its degree as in \cite{CFHR} through the Riemann--Roch formula
\[\deg\kf = \chi(\kf)-\chi(\ko_B).\]
\end{enumerate}
\end{defin}

\begin{rem}
In Definition \ref{def: restriction divisorial sheaf} we defined for a well-behaved divisor $A$ the sheaf $\ko_B(A)$ on a curve $B$. 
Unfortunately, if $A$ is $\IQ$-Cartier but not Cartier, then the degree of the  divisor on $B$ as in \refenum{ii} and  the degree of the  associated sheaf as in \refenum{iii} behave differently: the former may be rational while the latter is always an integer. This happens precisely because the restriction of divisorial sheaves to a curve is not multiplicative in general. 

We will mostly work with the numerical definition \refenum{ii} and try to make it clear when we need to consider torsion-free sheaves.
In some special situations, for example when $F=m(K_X + \Delta)$ and $B\subset \Delta$ on a stable log surface $(X,\Delta)$, a comparison between $\deg F\restr B$ and $\deg\ko_B(F)$ is possible (see Lemma~\ref{lem: QCar res}).
\end{rem}

\begin{rem}\label{rem: intersection properties}
The intersection form defined in this way has some unexpected properties: For example, if $A$ and $B$ are contained in different irreducible  components of $X$ then their intersection number  is zero even if they intersect in the non-normal locus. 
\end{rem}

\subsection{Descending pluri-log-canonical sections and  invariants}
Since we are especially interested in pluri-log-canonical maps and thus sections of pluri-log-canonical bundles, the following will play a role.
\begin{defin}[{\cite[5.11]{KollarSMMP}}]\label{def: different}
Let  $B$ be a well-behaved reduced curve on $X$ and  $B^\nu$  the normalisation of $B$. Suppose $\omega_X(\Delta + B)^{[m]}$ is a line bundle for some positive integer $m$. Then the different $\Diff_{B^\nu}(\Delta)$ is a $\IQ$-divisor on $B^\nu$ such that $m\Diff_{B^\nu}(\Delta)$ is integral and 
\[
 \omega_X(\Delta + B)^{[m]}\restr{B^\nu}\isom \omega_{B^\nu}^{[m]}(m\Diff_{B^\nu}(\Delta)).
\]
\end{defin}

\begin{prop}[{\cite[Prop.~5.8]{KollarSMMP}}]\label{prop: sections}
 Let $(X, \Delta)$ be a stable log surface and $m\geq1$ an integer. A section $s\in H^0(\bar X, \omega_{\bar X}(\bar D+\bar \Delta)^{[m]})$ descends to a section in $H^0(X, \omega_X(\Delta)^{[m]})$ if and only if its residue at the generic points of  $\bar D^\nu$ is  $\tau$-invariant if $m$ is even respectively $\tau$-anti-invariant if $m$ is odd.

If  $\omega_{\bar X}(\bar D+\bar \Delta)^{[m]}$ is a line bundle then this is equivalent to the image of $s$ in $H^0(\bar D^\nu, \omega_{\bar D^\nu}^{[m]}(m\Diff_{\bar D^\nu}(\Delta)))$ being $\tau$-invariant if $m$ is even respectively $\tau$-anti-invariant if $m$ is odd.
\end{prop}
 The alternating  signs are related to the Poincar\'e residue map: localising at a  codimension 1 nodal point we look at the local model $\IA^2\supset X = (xy=0)=L_x\cup L_y$ so that a local generator for $\omega_{\IA^2}(X)$ is $dx\wedge dy /xy$. Taking residues along the two lines we have
\[ \mathrm {Res}_{L_x}\left(\frac{dx\wedge dy}{xy}\right) = \frac{ dy}y, \quad\mathrm {Res}_{L_y}\left(\frac{dx\wedge dy}{xy}\right) = -\frac{  dx}x,
\]
so they differ in sign at the node.

For later reference we also state
\begin{prop}\label{prop: invariants}
 Let $X$ be a stable surface with normalisation $\bar X$. In the notation above
we have $K_X^2 = (K_{\bar X}+\bar D)^2$  and  $\chi(\ko_X) = \chi(\ko_{\bar X})+\chi(\ko_D)-\chi(\ko_{\bar D})$.
\end{prop}
\begin{proof}
 The first part is clear. 
 For the second note that the conductor ideal defines $\bar D$ on $\bar X$ and the non-normal locus $D$ on $X$. In particular, $\pi_*\ko_{\bar X}(-\bar D)=\ki_D$ and    additivity of the Euler characteristic for the two sequences
\begin{gather*}
 0\to \ko_{\bar X}(-\bar D)\to \ko_{\bar X}\to \ko_{\bar D}\to 0,\\
0\to \pi_*\ko_{\bar X}(-\bar D)\to \ko_{ X}\to \ko_{D}\to 0
\end{gather*}
gives the claimed result.
\end{proof}

\subsection{The curve embedding theorem}
The technique of restriction to curves will play a major role in our approach and thus we will often need the following numerical criterion due to Catanese, Franciosi, Hulek and Reid. We state it in a slightly weaker version, that suffices for our purpose.

\begin{theo}[{\cite[Thm.~1.1]{CFHR}}]\label{thm: curve embedding}
 Let $C$ be a projective curve (over an algebraically closed field) which is Cohen--Macaulay but not necessarily irreducible or reduced and $\kl$ a line bundle on $C$. Then $\kl$ is base-point-free if for every generically Gorenstein subcurve $B\subset C$
\[ \deg \kl\restr B = \chi(\kl\restr B)-\chi(\ko_B)  \geq 2p_a(B)=2(1-\chi(\ko_B))\]
and $\kl$ is very ample if the inequality is strict.
\end{theo}
Note that for an irreducible and smooth curve this gives the classical bounds.

\section{Some vanishing results} \label{section: vanishing}
We will need the following basic vanishing result, which is a variant of \cite[Cor.~6.6]{kss10},  and some consequences.
All of these results follow from general vanishing theorems in \cite{fujino12} but the surface case is technically much simpler.
\begin{prop}[Generalised Kodaira vanishing]\label{prop: Kodaira}
 Let $X$ be an slc surface and $A$ a well-behaved ample divisor on $X$. Then
\[H^i(X, \ko_X(-A)) = 0 \qquad\text{for all $i<2$.}\]
\end{prop}
\begin{proof}
 Choose $k\in \IN$ such that $\ko_X(kA)$ is a very ample line bundle and let $B$ be the divisor of a general section. In particular, $B$ is a reduced divisor contained in the locus where $X$ has at most  normal crossing singularities and $B$ is non-singular where $X$ is smooth. Now consider the ramified simple cyclic cover
\[ \pi\colon Y= \mathrm{Spec}_X\left(\bigoplus_{i=0}^{k-1} \ko_X(-kA)\right)\to X\]
with the following properties:
\begin{enumerate}
\item $Y$ is Cohen--Macaulay by construction since each $\ko_X(-kA)$ is $S_2$ and $Y$ is a surface.
\item $Y$ has Du Bois singularities by \cite[Cor.~6.21]{KollarSMMP}.
 \item $(\pi^*\ko_X(A))^{[1]}$ is locally free  \cite[2.44.6]{KollarSMMP} and ample because the pullback of an ample divisor via a finite map is ample.
\item Let $Z\subset X$ be a codimension 2 subset such that $X\setminus Z$ is Gorenstein. For every reflexive sheaf $\kf$ on $X$ we have
\[ \pi_*\left(\pi^*\kf\right)^{[1]}\isom \left(\pi_*\ko_Y\tensor_{\ko_X}\kf\right)^{[1]}\isom \bigoplus_{i=0}^{k-1} \left(\ko_X(-kA))\tensor_{\ko_X} \kf\right)^{[1]}\]
because all sheaves above are $S_2$ and isomorphic over $X\setminus Z$ where $\pi$ is flat.
\end{enumerate}

Now by \refenum{iii} the sheaf $\pi^*(\ko_X(-A))^{[1]}$ is a line bundle on $Y$  whose inverse is ample. Since $Y$ is Cohen--Macaulay and Du Bois, the Du Bois version of Kodaira vanishing \cite[Thm.~10.42]{KollarSMMP} implies for  $i<2$
\begin{align*}
0 &= H^{i }\left(Y, (\pi^*\ko_X(-A))^{[1]}\right)\\
&=H^{i }\left(X, \pi_*\left((\pi^*\ko_X(-A))^{[1]}\right)\right)\\
& =H^{i }\left(X,  \bigoplus_{i=0}^{k-1} \left(  \ko_X(-kA)) \tensor_{\ko_X} \ko_X(-A)   \right)^{[1]}\right) \quad \text{ by \refenum{iv}}\\
& =H^{i }\left(X,  \bigoplus_{i=0}^{k-1} \ko_X((-1-k)A)\right)\\
&\supset H^{i}\left(X,   \ko_X(-A)\right)
\end{align*}
This concludes the proof.
\end{proof}
\begin{rem}
 The reason for the restriction to dimension 2 in the above theorem is that the index-1-cover constructed in the proof may well fail to be Cohen--Macaulay in higher dimensions, which is needed to apply the Du Bois version of Kodaira vanishing.
\end{rem}

\begin{lem}\label{lem: ext reflexive}
 Let $X$ be a demi-normal surface and $\kf, \kg$ reflexive sheaves on $X$ that are locally free outside codimension 2. Then $\shext^1_{\ko_X}(\kf, \kg) =0$. In particular,  the local-to-global $\Ext$-spectral sequence induces isomorphisms 
  \[H^0(X, \shom_{\ko_X}(\kf, \kg))\isom  \Hom_X(\kf, \kg) \text{ and } H^1(X, \shom_{\ko_X}(\kf, \kg))  \isom \Ext^1_X (\kf , \kg) .\]
\end{lem}
\begin{proof}
The first isomorphism  is clear. The second  sits in an exact sequence
\[0\to H^1(X, \shom_{\ko_X}(\kf, \kg))\to \Ext^1_X(\kf, \kg) \to H^0(X, \shext^1_{\ko_X}(\kf, \kg)) \]
so it is sufficient to show that $\shext^1_{\ko_X}(\kf, \kg) =0$.

To see this, let $j\colon U \into X$ be the inclusion of an open subset where all sheaves in question are locally free and such that the complement is of codimension 2. Then  by \cite[Ch.~2, Prop.~5.8]{HartshorneRD} 
\[j^*\shext^1_{\ko_X}(\kf, \kg)=\shext^1_{\ko_U}(j^*\kf, j^*\kg)=0\]
so that $\shext^1_{\ko_X}(\kf, \kg)$ is torsion supported in codimension 2.

Therefore, to study $\shext^1_{\ko_X}(\kf, \kg)$ we may assume that $X$ is affine. Since both $\kf$ and $\kg$ are reflexive, in any extension $ 0\to \kg\to \ke\to \kf\to 0$ the sheaf $\ke$ is also reflexive. Thus the extension is determined outside codimension 2 as can be seen from the diagram
\[
\begin{tikzcd}
   0\rar & \kg\arrow[equal]{d}\rar& \ke\dar\rar& \kf\arrow[equal]{d}\rar& 0\\
0\rar & j_*(j^*\kg)\rar & j_*(j^*\ke)\rar& j_*(j^*\kf)\rar& 0.
\end{tikzcd}
\]

Thus the extension splits if it splits outside codimension 2 and $\shext^1_{\ko_X}(\kf, \kg)$ has no torsion supported in codimension 2. By the above it is actually zero.
\end{proof}

\begin{cor}\label{cor: adj van}
 Let $(X, \Delta) $ be a stable log surface. Then for all $i>0$ and all integers $m\geq 2$,
 we have $H^i(X, \omega_X^{[m]}((m-1)\Delta))=0$.
\end{cor}

\begin{proof}
By Serre duality on the Cohen--Macaulay scheme $X$ and Lemma~\ref{lem: ext reflexive}, we have
\begin{align*}
 H^i\left(X, \omega_X^{[m]}((m-1)\Delta)\right)^*&\isom \Ext^{2-i}_X(\omega_X^{[m]}((m-1)\Delta), \omega_X)\\
 &\isom H^{2-i } (X, \shom_{\ko_X}(\omega_X^{[m]}((m-1)\Delta), \omega_X))\\
&\isom H^{2-i } (X, \omega_X(\Delta)^{[1-m]}),
\end{align*}
where in the last step we used the fact that $\shom_{\ko_X}(\omega_X^{[m]}((m-1)\Delta), \omega_X)$ and $\omega_X(\Delta)^{[1-m]}$ are both reflexive and coincide outside codimension 2. So by  Proposition~\ref{prop: Kodaira},  $\omega_X^{[m]}((m-1)\Delta)$ has no higher cohomology for $m\geq 2$ which proves the claim.
\end{proof}

\begin{cor}\label{cor: vanishing2}
 Let $(X, \Delta) $ be a stable log surface, $\ki$ the ideal sheaf of $D\cup\Delta\subset X$ and $\ki(m(K_X+\Delta))=(\ki\tensor \omega_X(\Delta)^{[m]})^{[1]}$. In our standard notation \ref{notation} we have
\begin{enumerate}
 \item $H^0(X, \ki(m(K_X+\Delta))) = H^0(\bar X, \omega_{\bar X}(\bar D+\bar\Delta)^{[m]}(-\bar D-\bar\Delta))$ for $m\geq 1$, 
 \item $H^i(X, \ki(m(K_X+\Delta)))=0$ for all $i>0$ and all integers $m\geq 2$.
\end{enumerate}
\end{cor}
\begin{proof}
Let $\bar\ki$ be the ideal sheaf of $\bar D\cup\bar\Delta$ in $\bar X$. Then we have $\pi_*\bar\ki = \ki$. Let $U$ be the subset of $X$ where $K_X+\Delta$ is Cartier, whose complement has codimension at least 2, and ${\bar U}=\inverse f( U)$. Then over $U$ we can use the projection formula to obtain 
\begin{multline*}
\pi_*\left(\omega_{\bar X}(\bar D+\bar\Delta)^{[m]}(-\bar D-\bar\Delta)\restr{\bar  U}\right)  \isom  \pi_*\left(\bar\ki\otimes\pi^*\omega_X(\Delta)^{[m]}\right)\restr{\bar U}\\
\isom  \pi_*\bar\ki\restr{\bar U}\otimes\omega_X(\Delta)^{[m]} \restr U\isom  \ki\omega_X(\Delta)^{[m]}\restr U.
\end{multline*}
Because all sheaves in question are $S_2$ this extends to an isomorphism 
\[\pi_*\omega_{\bar X}(\bar D+\bar\Delta)^{[m]}(-\bar D-\bar\Delta)\isom \ki(m(K_X+\Delta)).
\]
To conclude we use that $\pi$ is a finite morphism and thus 
\[H^i(X,\ki(m(K_X+\Delta))) = H^i(\bar X, \omega_{\bar X}(\bar D+\bar\Delta)^{[m]}(-\bar D-\bar\Delta) ) \text{ for } i
\geq 0.\]
In particular \refenum{i} is proved. By applying  Corollary~\ref{cor: adj van}  to the pair $(\bar X, \bar D+\bar\Delta)$, the second item also follows. 
\end{proof}
The first item of Corollary~\ref{cor: vanishing2} is also a direct consequence of Proposition \ref{prop: sections}.

The next proposition is much easier to prove if $m$ is a multiple of the index, but we need this strong form for the result about the log-canonical-ring in Section \ref{sect: canonical ring}. It is actually identical to the much easier Corollary \ref{cor: adj van} if $\Delta=0$.
\begin{prop}\label{prop: vanishing}
 Let $(X, \Delta) $ be a stable log surface. Then for all $i>0$ and all integers $m\geq 2$,
we have $H^i(X, \omega_X(\Delta)^{[m]})=0$.
\end{prop}
The idea of the proof is to use the restriction sequence (Lemma~\ref{lem: div seq}) together with vanishing on a curve and the vanishing results proved above. Due to the singularities 
one has to take  extra care. We begin with some preliminary results.

We consider the minimal log-semi-resolution  $f\colon Y \to X$ (see remarks after Definition~\ref{defin: semi resolution}) and the respective normalisations  
$\pi \colon \bar X\to X$ and $\mu\colon \bar Y\to Y$. In particular,  $\Delta_Y= (\inverse f)_*\Delta$ and $D_Y$, the conductor of $Y$ are disjoint and hence $\Delta_Y$ is contained in the smooth locus of $Y$. Obviously, there is a birational morphism $\bar f\colon \bar Y\rightarrow \bar X$ such that $\pi\circ\bar f = f \circ\mu$ (cf.~diagram \eqref{equation: normalisation+resolution} in the appendix).
\begin{lem}\label{lem: R1 van}
For any well-behaved divisor $M$ on $X$, one has
\[
 R^1f_*\ko_Y(K_Y + \lceil f^*M\rceil) = 0.
\]
\end{lem}
\begin{proof}
 Let $\bar M$ be the strict transform of $M$ on $\bar X$. Then $\bar f^*\bar M$ is an $\bar f$-nef divisor on $\bar Y$. By \cite[2.2]{sak84} we have $R^1\bar f_*\ko_{\bar Y}(K_{\bar Y} + \lceil \bar f^*\bar M\rceil) = 0$. Now consider the following exact sequence
\begin{equation}\label{seq: R1 van}
 0 \rightarrow \mu_*\ko_{\bar Y}(K_{\bar Y} + \lceil  \bar f^*\bar M\rceil) \rightarrow \ko_Y(K_Y + \lceil f^* M\rceil) \rightarrow \kq \rightarrow 0
\end{equation}
where the quotient $\kq$ is supported on the conductor divisor $D_Y$. Applying $f_*$ to \eqref{seq: R1 van} we obtain an exact sequence
\begin{equation}\label{eq: R1 van}
 R^1f_*\mu_*\ko_{\bar Y}(K_{\bar Y}+ \lceil \bar f^*\bar M\rceil)\rightarrow  R^1f_*\ko_Y(K_Y + \lceil f^* M\rceil) \rightarrow  R^1f_*\kq.
\end{equation}
Since $\mu\colon \bar Y\rightarrow \bar X$, $\pi\colon \bar X\rightarrow X$ and $f\restr {D_Y}$ are all finite morphisms, one sees easily that $R^1f_*\kq= 0$ and $R^1f_*\mu_*\ko_{\bar Y}(K_{\bar Y}+ \lceil \bar f^* \bar M\rceil)=\pi_* R^1\bar f_*\ko_{\bar Y}(K_{\bar Y} + \lceil \bar f^*\bar M\rceil) =0$. The lemma follows now from \eqref{eq: R1 van}.
\end{proof}

\begin{lem}\label{lem: QCar res}
\begin{enumerate}
 \item The torsion free sheaf $\ko_\Delta (m(K_X+\Delta))$ is a line bundle for any positive integer $m$. 
\item Let $\nu\colon \Delta^\nu \rightarrow \Delta$ be the normalisation of $\Delta$ and $\Diff_{\Delta^\nu}(0)=\sum a_p p$ the different. For $m\geq 2$ we have an inclusion of line bundles 
\[\omega^{\otimes m}_\Delta\left(\sum_{\nu(p)\in\Delta_\mathrm{sm}} \lceil (m-1)a_p\rceil \nu(p)\right)\subset \ko_\Delta (m(K_X+\Delta)),\]
where $\Delta_\mathrm{sm}$ denotes the smooth locus of $\Delta$. 
\end{enumerate}
\end{lem}
\begin{proof}
First we prove that, if $p\in \Delta$ is singular then $K_X+\Delta$ is Cartier at $p$. Note that $\Delta$ has at most ordinary nodes as singularities and the resolution graph of $p\in X$ is a chain of analytic irreducible components of type $\mathrm{(C2)}$ in Section~\ref{sect: semi-rat}. Let $E\subset Y$ be the (whole) reduced exceptional divisor over $p$. Then it is easy to see that the divisor $K_Y+\Delta_Y + E$ is Cartier in a neighbourhood of $E$. Using the same argument as \cite[Lemma~1.1]{shepherd-barron83}  an isomorphism $\ko_Y(K_Y+\Delta_Y + E)\cong\ko_Y$ in an analytic neighbourhood of $E$ can be established and consequently $K_X + \Delta$ is Cartier at $p$. 

So the rank one torsion free sheaf $\ko_\Delta (m(K_X+\Delta))$ is locally free at the singular points of $\Delta$ and hence is locally free everywhere. This proves \refenum{i}.

On the open subset $U$ where  $K_X+\Delta$ is Cartier the residue map induces a canonical isomorphism $\kr_{U\cap\Delta}\colon\omega_X(\Delta)^{[m]}\restr{U\cap \Delta}\isom \omega_{U\cap \Delta}^{\tensor m}$ \cite[(4.1.3)]{KollarSMMP}. Viewing the residue map $\kr_{U\cap\Delta}$ as a rational section of $\shom(\omega_X(\Delta)^{[m]}\restr\Delta,\omega_\Delta^{\tensor m})$ we obtain a natural isomorphism 
 \[\kr_\Delta\colon\ko_\Delta(m(K_X+\Delta)) \cong \omega_\Delta^{\tensor m}(\sum_q b_q q)\]
 where $\sum_q b_q q$ is a uniquely determined integral divisor supported on those smooth points of $\Delta$ where the different is non-zero (see \cite[Sect.~4.1]{KollarSMMP}). The integers $b_q$ only depend on what happens in  a neighbourhood of $q\in X$. 

Now we will establish the inequality $b_q\geq\lceil(m-1)a_{\nu^{-1}(q)}\rceil$ for $q\in\Delta_\mathrm{sm}$.

Recall that $f\colon Y\rightarrow X$ is the minimal semi-log-resolution, so that $\Delta_Y$ and $D_Y$ are disjoint. We choose a well-behaved integral canonical divisor $K_X$ and let $M=(m-1)(K_X+\Delta)$  (cf.~Section~\ref{sect: divisors}). Applying $f_*$ to the following exact sequence
\begin{multline*}
 0\rightarrow \ko_Y(K_Y + \lceil f^*M\rceil) \rightarrow  
\ko_Y(K_Y + \Delta_Y + \lceil f^* M\rceil) \\ 
\rightarrow\ko_{\Delta_Y}(K_Y + \Delta_Y +\lceil f^*M\rceil) \rightarrow 0.
\end{multline*}
we obtain by Lemma~\ref{lem: R1 van}
\begin{multline*}
 0\rightarrow f_*\ko_Y(K_Y + \lceil f^*M\rceil) \rightarrow  
f_*\ko_Y(K_Y + \Delta_Y + \lceil f^* M\rceil) \\ 
\rightarrow f_*\ko_{\Delta_Y}(K_Y + \Delta_Y +\lceil f^*M\rceil) \rightarrow 0.
\end{multline*}
The identification of the sheaves on semi-log-smooth locus of $X$ yields natural vertical morphisms in the following diagram
\[
 \begin{tikzcd}
  f_*\ko_Y(K_Y + \lceil f^*M\rceil)\rar\dar&f_*\ko_Y(K_Y + \Delta_Y + \lceil f^* M\rceil)\dar\\
\ko_X(K_X + M) \rar & \ko_X(K_X + \Delta +  M).
 \end{tikzcd}
\]
This induces a morphism between the quotients of the horizontal arrows, namely, 
\[\phi\colon f_*\ko_{\Delta_Y}(K_Y + \Delta_Y +\lceil f^*M\rceil) \rightarrow \ko_\Delta(K_X + \Delta +  M),\] which is isomorphism at the generic points of $\Delta$ and in particular injective. For $\Delta_Y$, which lies in the smooth locus of $Y$, we have a natural identification induced by residue map
\begin{equation*}
 \begin{split}
 \ko_{\Delta_Y}(K_Y + \Delta_Y +\lceil f^*M\rceil) &= \ko_{\Delta_Y}(m(K_Y + \Delta_Y))  \otimes \ko_{\Delta_Y}(\sum_p\lceil(m-1)a_p\rceil p) \\
&=\omega_{\Delta_Y}^{\otimes m}(\sum_p\lceil(m-1)a_p\rceil p).
\end{split}
\end{equation*}
As indicated before, for a smooth point $q$ of $\Delta$ we can focus on a neighbourhood of $q\in X$ and consequently obtain a comparison  $b_q\geq \lceil(m-1)a_{\nu^{-1}(q)}\rceil$ by the injectivity of $\phi$ at $q$. This finishes the proof of (ii).
\end{proof}

\begin{proof}[Proof of Proposition \ref{prop: vanishing}]
Consider  the exact sequence
\[0\to \omega_X^{[m]}((m-1)\Delta) \to \omega_X(\Delta)^{[m]}\to \ko_\Delta(m(K_X+\Delta))\to 0, \]
whose long exact cohomology sequence gives, using Corollary \ref{cor: adj van}, 
\begin{gather}\label{eq: cohomology}
 H^1(X, \omega_X(\Delta)^{[m]})\isom H^1(\Delta, \ko_\Delta(m(K_X+\Delta))), \qquad H^2(X, \omega_X(\Delta)^{[m]}) =0.
\end{gather}
It remains to show that $H^1(\Delta, \ko_\Delta(m(K_X+\Delta))) =0$ for $m\geq 2$.
We argue by contradiction, so assume $H^1(\Delta,\ko_\Delta(m(K_X+\Delta)))\neq 0$. By Lemma~\ref{lem: H^1 on curve} there is a subcurve $B\subset\Delta$ with a generically onto morphism $$\lambda_B\colon\ko_B(m(K_X+\Delta)) \rightarrow \omega_B.$$ 
On the other hand, one sees by Lemma~\ref{lem: QCar res} that $\ko_B(m(K_X+\Delta))=\ko_\Delta(m(K_X+\Delta))\restr B$ has degree strictly larger than $\deg \omega_B$. This is a contradiction.\end{proof}

\section{Base-point-freeness of pluri-log-canonical maps}\label{section: bpf}
Many of the standard techniques do not work directly on non-normal and possibly reducible surfaces. Thus to prove base-point-freeness we first use a Reider-type result of Kawachi on the normalisation to produce pluri-log-canonical sections vanishing along the non-normal locus. Then we analyse the restriction of the pluri-log-canonical bundle to the non-normal locus directly. The overall result is

\begin{theo}\label{thm: base-point-free}
Let $(X, \Delta)$ be a connected stable log surface of global index $I$.
\begin{enumerate}
\item The line bundle $\omega_X(\Delta)^{[mI]}$  is base-point-free for $m\geq 4$.
\item The line bundle $\omega_X(\Delta)^{[mI]}$  is base-point-free for $m\geq 3$  if one of the following holds:
\begin{enumerate}
\item $I\geq 2$.

\item There is no irreducible component $\bar X_i$ of the normalisation such that $\left(\pi^*(K_X+\Delta)\restr {\bar X_i}\right)^2=1$, and the union of $\Delta$ and the non-normal locus is a nodal curve.
\item $X$ is normal and we do not have $I=(K_X+\Delta)^2=1$.
\end{enumerate}
\end{enumerate}
\end{theo}

This result is sharp in the sense that there are examples of smooth surfaces such that $\omega_X^{\tensor 3}$ has base-points \cite[Remark on p.287]{BHPV}. However, the conditions in the theorem are by no means necessary for base-point-freeness.

In the case of surfaces with canonical singularities the bi-canonical map is a morphism as soon as $K_X^2\geq 5$. We will show in   Example~\ref{exam: large K^2} that this does not generalise to stable surfaces, even irreducible ones.

\begin{rem}\label{rem: finitemorphism}
 If  $|mI(K_X+\Delta)|$ is base-point-free, the fact that $K_X$ is ample implies that the pluri-log-canonical map $\phi_{mI}\colon X\rightarrow \IP^N$ does not contract any curve on $X$, hence defines a finite morphism from $X$ to its reduced image. 
\end{rem}

\subsection{Base-point-freeness on the complement of boundary and  non-normal locus}
We now analyse the base-point-freeness of pluri-log-canonical maps outside the non-normal locus and the boundary of a stable log surface, starting with the following auxiliary result.
\begin{prop}\label{prop: bpfaux}
 Let $(\bar X, \bar \Delta ) $ be an irreducible (hence connected) log surface with log-canonical singularities, $K_{\bar X}+\bar \Delta$ ample and  $I$ the global index. Assume $m\geq 3$. If $I= 1$ assume  in addition $(m-1)^2 (K_{\bar X}+\bar \Delta)^2>4$.

 Then for every $x\in \bar X\setminus \bar\Delta$ there is a section of $\omega_{\bar X} (\bar\Delta)^{[mI]}(-\bar \Delta)$ not vanishing at $x$. In particular,  the rational map associated to $\omega_{\bar X} (\bar\Delta)^{[mI]}$ is a morphism on $\bar X\setminus \bar \Delta$.
\end{prop}
\begin{proof}
Let $\bar f\colon  \bar Y \rightarrow \bar X$ be the  minimal resolution of those singularities of $\bar X$ that are contained in $\bar\Delta$. 

Fix once for all a Weil divisor $K_{\bar Y}$ representing the canonical class of $\bar Y$ and a Weil divisor $ M$ on $\bar Y$ such that $\ko_{\bar Y}(M) = \bar f^*\omega_{\bar X}(\bar \Delta)^{[mI]}$. Let $\Delta_{\bar Y}$ be the strict transform of $\bar\Delta$. We can write
\[K_{\bar Y} +\Delta_{\bar Y} + E \sim_\IQ\bar f^*(K_{\bar X}+\bar \Delta)\sim_\IQ\frac{1}{mI} M\]
where $\sim_\IQ$ denotes $\IQ$-linear equivalence and $E$ is a $\IQ$-divisor supported on the exceptional locus, effective since $\bar f$ was chosen to be minimal.

With these choices $M - K_{\bar Y} -\Delta_{\bar Y} - E$ is a $\IQ$-Weil divisor such that 
\begin{equation}\label{eq: qlin}
M - K_{\bar Y} -\Delta_{\bar Y} - E\sim_\IQ(mI-1)\bar  f^*(K_{\bar X}+\bar \Delta)
\end{equation}
and hence big and nef. In addition
\begin{align*}
   K_{\bar Y}+ \lceil   M - K_{\bar Y} -\Delta_{\bar Y} - E\rceil 
& = \lceil K_{\bar Y} +M  - K_{\bar Y} -\Delta_{\bar Y} - E\rceil \\
                                 & = M  - \Delta_{\bar Y} - \lfloor E\rfloor
  \end{align*}
is Cartier because $\bar Y$ is smooth near $\Delta_{\bar Y}\cup E$.

Thus \cite[Thm.~2]{kawachi00} implies that $K_{\bar Y} + \lceil M-K_{\bar Y}-\Delta_{\bar Y}-E\rceil$ is base-point-free outside the exceptional locus of $\bar f$ as soon as the numerical conditions (using  \eqref{eq: qlin})
\begin{equation}\label{eq: ineq} 
\left((mI-1)\bar  f^*(K_{\bar X}+\bar \Delta)\right)^2 >4\text{ and }(mI-1)\bar  f^*(K_{\bar X}+\bar \Delta)C \geq 2 
\end{equation}
hold for all curves $C$ not  contained in the exceptional locus of $\bar f$. 

Assuming this for a moment we see that 
\[
 \begin{split}
H^0 (\bar Y, \ko_{\bar Y}( K_{\bar Y}+ \lceil M-K_{\bar Y}-\Delta_{\bar Y}-E \rceil))
 &=H^0 (\bar Y, \ko_{\bar Y}(M   - \Delta_{\bar Y} - \lfloor E\rfloor)) \\
&\subset  H^0(\bar Y,  \ko_{\bar Y}(M- \Delta_{\bar Y}))\\
& \isom  H^0(\bar X, \bar f_*\ko_{\bar Y}(M  - \Delta_{\bar Y}))\\
& =H^0(\bar X,  \omega_{\bar X} (\bar\Delta)^{[mI]}(-\bar \Delta)).\\
\end{split}
\]
Thus 
all sections of  $K_{\bar Y} + \lceil M-K_{\bar Y}-\Delta_{\bar Y}-E \rceil$ descend to sections of  $\omega_{\bar X}(\bar\Delta)^{[mI]}(-\bar \Delta)$. Via the inclusion $\omega_{\bar X}(\bar\Delta)^{[mI]}(-\bar \Delta)\into \omega_{\bar X}(\bar\Delta)^{[mI]}$ we can also interpret these as section of $\omega_{\bar X}(\bar\Delta)^{[mI]}$ vanishing along $\bar\Delta$.  The restriction $\bar f\colon \bar Y \setminus \inverse{\bar f} \bar \Delta\to \bar X\setminus \bar\Delta$ is an isomorphism  and thus base-point-freeness  holds on $\bar X\setminus \bar \Delta$  under the assumption \eqref{eq: ineq}.

It remains to show that \eqref{eq: ineq} holds under the assumptions of the proposition. First we note that for $m\geq 3$ and for every non-exceptional curve $C\subset \bar Y$ we have 
\begin{multline*}(mI-1)\bar  f^*(K_{\bar X}+\bar \Delta)   C=
 (mI-1)(K_{\bar X} +\bar \Delta)\bar f_* C \\= 2I(K_{\bar X} +\bar \Delta)\bar f_*C + ((m-2)I-1)(K_{\bar X} +\bar \Delta)\bar f_*C\geq 2
\end{multline*} 
because $I(K_{\bar X}+\bar \Delta)$ is an ample Cartier divisor.

Noting that $I(K_{\bar X}+\bar \Delta)^2$ is a positive integer, because it is the intersection of an ample Cartier divisor with an integral divisor, and  writing
\[
(Im-1)^2(K_{\bar X} +\bar \Delta)^2>4 \iff \left(m - \frac 1 I\right)^2 I(K_{\bar X}+\bar \Delta)^2>\frac 4 I\]
we see that if $I\geq 2$  and $m\geq 2 $ the inequality is satisfied. If $I = 1$ we need that $(m-1)^2 (K_{\bar X}+\bar \Delta)^2>4$ in addition to $m\geq 3$. 
\end{proof}
We now descend the above result to a possibly non-normal stable log surface.
\begin{cor}\label{cor: bpfoutsideD}
 Let $(X,\Delta)$ be a stable log surface with global index $I$. Then the base-points of $\omega_X(\Delta)^{[mI]}$ are contained in the union of the non-normal locus $D$ and the boundary $\Delta$ if
\begin{enumerate}
\item $m\geq4$, 
\item $m\geq3$ unless the index $I=1$, and there is an irreducible component $\bar X_i$ of the normalisation such that $\left(\pi^*(K_X+\Delta)\restr {\bar X_i}\right)^2=1$.
\end{enumerate}
\end{cor}
\begin{proof}
 We use our standard notation \ref{notation}.
On every irreducible component $\bar X_i$ of the normalisation $\bar X$, we apply Proposition~\ref{prop: bpfaux} to the pair $(\bar X_i, (\bar \Delta+\bar D)\restr{\bar X_i})$ which, under our assumptions, gives for every point $\bar x\in \bar X\setminus (\bar\Delta\cup \bar D)$ a section of $\omega_{\bar X}(\bar D + \bar\Delta)^{[mI]}(-\bar D -\bar\Delta)$ not vanishing at $\bar x$.

Via the inclusion 
\[ \omega_{\bar X}(\bar D + \bar\Delta)^{[mI]}(-\bar D - \bar\Delta)\into \omega_{\bar X}(\bar D + \bar\Delta)^{[mI]},\]
 the sections of $\omega_{\bar X}(\bar D + \bar\Delta)^{[mI]}(-\bar D -\bar\Delta)$ are mapped to sections of $ \omega_{\bar X}(\bar D + \bar\Delta)^{[mI]}$ that vanish along $\bar D\cup \bar\Delta$ and thus descend to sections  of $\omega_X(\Delta)^{[mI]}$ by Proposition~\ref{prop: sections}. Consequently, $\omega_X(\Delta)^{[mI]}$ has no base-points outside $\Delta\cup D$.
\end{proof}

\subsection{Restricting to non-normal locus and boundary}\label{sect: non-normal locus}
In this section we concentrate on the geometry of the non-normal locus $D$ of a stable log surface $(X, \Delta)$, using the same notation \ref{notation} as always. We start with a definition that will turn out to describe all possible singularities of the non-normal locus.
\begin{defin}\label{defin: multinode}
 Let $B$ be a reduced curve, $p\in B$ and $\mu = \mu_p(B)$ the multiplicity of $p$ in $B$. We call $p$ a $\mu$-multi-node, if locally analytically at $p$ the curve is isomorphic to the union of the coordinate axes in $\IA^\mu$.
\end{defin}
Thus a 1-multi-nodal point is smooth and a 2-multi-node is just an ordinary nodal point.

By the correspondence between stable log surfaces and their normalisation explained in  Theorem~\ref{thm: triple}  and Notation~\ref{notation}, the non-normal locus $D$ is quotient by the finite equivalence relation on $\bar D$ induced by $\tau$. In other words, as a set $D$ is the set of equivalence classes of the equivalence relation on $\bar D^\nu$ generated by 
\begin{equation}\label{eq: eqrel}
 \bar p\sim\bar  q \text{ if } \tau(\bar p)=\bar q \text{ or } \nu(\bar p) = \nu (\bar q), 
\end{equation}
and the scheme structure on $ D$ is   determined by the requirement that the diagram
\[
 \begin{tikzcd}
\bar D  \dar{\pi}& \bar D^\nu\lar[swap]{\bar \nu}\dar{/\tau}\\  
D & D^\nu \lar[swap]{\nu}
 \end{tikzcd}
\]
is a pushout diagram.

Recall that $\bar D$ is a nodal curve by the classification of log-canonical singularities.
If $\bar D$ is smooth then $\bar \nu$ is an isomorphism, so $D^\nu=\bar D/\tau$ satisfies the pushout property and $D=D^\nu$. Applying the same argument locally, it follows that if $p\in D$ is a point such that $\inverse \pi (p)$ contains only smooth points of $\bar D$ then $p$ itself is smooth.

Now let $p$ be a singular point of $D$. Write the preimage  of $p$ in $\bar D^\nu$ as
\begin{gather*}
\inverse{(\pi\circ \bar \nu)}(p)= \{a_1, b_1, \dots, a_k, b_k, c_1, \dots, c_l\}
\end{gather*}
 such that $\bar\nu(a_i)=\bar\nu(b_i)$ is a node of $\bar D$ and $\bar\nu(c_j)$ is a smooth point of $\bar D$. By the above there is at least one node of $\bar D$ mapping to $p$, so $k\geq 1$. Since the preimage of $p$ in $\bar D^\nu$ is an equivalence class with respect to the relation generated by \eqref{eq: eqrel}, we have $l\in \{0,1,2\}$ and, up to reordering, the following cases can occur (compare also \cite[17.4]{kollar12}):
\begin{description}
 \item[Type I] The preimage of $p$ in $\bar D$ consists only of nodes and we glue the normalisation $\bar D^\nu$ in a circular fashion: $\tau(b_i) = a_{i+1}\, (i=1, \dots, k-1),\, \tau(b_k)=a_1$.
If $k=1$, then $p$ is a smooth point of $D$, so we may assume $k\geq 2$.
\item[Type II] We glue the preimages of the nodes of $\bar D$ in a chain  $\tau(b_i) = a_{i+1}\, (i=1, \dots, k-1)$ and have two remaining points $a_1$ and $b_k$ at the ends. At each end we have two possibilities which we spell out only for $a_1$: either $\tau(a_1)=a_1$ or $\tau(a_1)=c_j$ for a point $c_j$ mapping to a smooth point of $\bar D$.
\end{description}

To determine the local structure of $D$ at $p$, we replace $D$ by a small analytic neighbourhood of $p$ such that $p$ is the only singular point and $\tau$ has no fixed points on $\bar D^\nu\setminus \inverse{(\pi\circ \bar \nu)}(p)$. Then, possibly shrinking $D$ further, the normalisation  $D^\nu = \bar D^\nu/\tau$ consists of $k$ (Type I) or $k+1$ (Type II) branches, each containing a unique point that maps to $p$, and $\bar D$ consists of $k$ nodal  and possibly one or two smooth branches for Type II.

Thus there are maps from $\bar D$ and $D^\nu$ to a neighbourhood of a multi-nodal point compatible with the equivalence relation. These satisfy the pushout conditions because if the tangent directions were not independent then the map to $D$ would factor over the multi-nodal point. So every singular point $p\in D$ is a $\mu$-multi-nodal point for some $\mu\geq 2$.

By Theorem~\ref{thm: triple} the different $\Diff_{\bar D^\nu}(\bar \Delta)$ is $\tau$-invariant and thus has the same coefficient for each point in an  equivalence class of the relation generated by \eqref{eq: eqrel}. In particular, if $p\in D$ is singular then the preimage contains at least one node and the different has coefficient 1 at each point mapping to $p$. This restricts the possibilities for the smooth points $\bar\nu(c_i)\in \bar D$ occurring in Type II: by the classification of log-canonical singularities they are either dihedral quotient singularities (see \cite[16.3]{kollar12}) of $\bar X$ or smooth points of $\bar D$ where $\bar D$ intersects $\bar \Delta$.

The first two items of the next lemma sum up the discussion so far.
\begin{lem}\label{lem: multinode}
 Let $(X, \Delta)$ be  a log surface with slc singularities and $D\subset X$ the non-normal locus. 
 Let $p\in D\cup \Delta$. Then the following holds.
\begin{enumerate}
 \item  $p$ is a $\mu$-multi-node of $D\cup \Delta$ for  $\mu=\mu_p(D\cup\Delta)\geq 1$. 
 \item If $p$ is a singularity of $D\cup\Delta$ then the inverse image $\pi^{-1}(p)$ contains at least one node  of  $\bar D\cup \bar \Delta$ and  at each point of $\bar D^\nu$ mapping to $p$ the different  $\Diff_{\bar D^\nu}(\bar \Delta)$  has coefficient 1.
 \item Let $B\subset D\cup\Delta$ be a subcurve. Then $p_a(B) =  p_a(B^\nu) + \sum_p (\mu_p(B)-1)$.
\end{enumerate}
\end{lem}
\begin{proof}
 It remains to prove the last item. We compare the arithmetic genera of $B$ and $B^\nu$ by taking Euler characteristics in the short exact sequence
 \[ 0\to \ko_B \to \nu_*\ko_{B^\nu} \to \nu_*\ko_{B^\nu}/\ko_B\to 0.\] Since all singular points are $\mu$-multi-nodes the length of the subsheaf of $\nu_*\ko_{B^\nu}/\ko_B$ supported at the point $p$ is exactly $\mu_p(B)-1$, so $p_a(B) =  p_a(B^\nu) + \sum_p (\mu_p(B)-1)$.
\end{proof}

Now we analyse the restriction of the log-canonical divisor to the non-normal locus and the boundary.
\begin{lem}\label{lem: deg comparison}
 Let $(X, \Delta)$ be a stable log surface with normalisation $(\bar X, \bar D + \bar \Delta)$.  Consider a subcurve $B\subset D\cup\Delta$, with normalisation $ B^\nu$. 
Let $s$ be the number of smooth points of $B$ that are singular points of $D\cup \Delta$.
  Then
\begin {align*}
\deg (K_X+\Delta)\restr{B^\nu} &\geq 2p_a(B^\nu) -2 + \sum_{p\in\sing {(D\cup \Delta)}}  \mu_p(B)\\
& \geq 2p_a(B) -2 + \sum_{p\in\sing B} (2- \mu_p(B))+s\\
& \geq 2p_a(B) -2 +\sum_{p\in\sing B} (2- \mu_p(B)).
\end {align*}
\end{lem}
\begin{proof}
Let $B_1 = B\cap D$ and $B_2 = B\cap \Delta$, that is, $B_1$ (resp.\ $B_2$) is the subcurve of $B$ contained in $D$ (resp.\ $\Delta$). As Weil divisors we can write $D+\Delta = A +B = A_i + B_i$ with $A$ (resp.\ $A_i$) being the complement curve of $B$ (resp.\ $B_i$). We adopt our usual notation for strict transforms in $\bar X$ and normalisation: for example, $\bar B$ is the strict transform of $B$ in $\bar X$, $B^\nu$ is the normalisation of $B$ while $\bar B^\nu$ is the normalisation of $\bar B$. 

Note that  $\pi_1\colon\bar B_1^\nu\rightarrow  B_1^\nu$ is a double cover and $\pi_2\colon\bar B_2^\nu \rightarrow B_2^\nu$ is an isomorphism. Thus by Hurwitz formula 
 \begin{equation}\label{eq: hurwitz}
  K_{\bar B_1^\nu} = \pi^* K_{B_1^\nu} + R \text{ and } K_{\bar B_2^\nu} = \pi^* K_{B_2^\nu} 
 \end{equation}
 where $R$ is the (reduced) ramification divisor on $\bar B_1^\nu$.
Now we compute the degree of $K_X+\Delta$ restricted to $B$:
\begin{align}\label{eq: decompose}
\deg  (K_X+\Delta)\restr{B} &=\deg (K_X+\Delta)\restr{ B_1}+ \deg (K_X+\Delta)\restr{B_2} \notag\\
&=\frac{1}{2} \deg (K_{\bar X}+\bar D+\bar\Delta)\restr{\bar B_1^\nu} + \deg (K_{\bar X}+\bar D+\bar\Delta)\restr{\bar B_2^\nu}\notag\\
  &  = \frac{1}{2}\deg (K_{\bar B_1^\nu}+\Diff_{\bar B_1^\nu}(\bar A_1))+ \deg (K_{\bar B_2^\nu}+\Diff_{\bar B_2^\nu}(\bar A_2)) \notag\\
& \overset{\eqref{eq: hurwitz}}= \deg \left(\frac{1}{2}\left(\pi_1^*K_{ B_1^\nu}+R +\Diff_{\bar B_1^\nu}(\bar A_1)\right)+ \pi_2^*K_{B_2^\nu} +\Diff_{\bar B_2^\nu}(\bar A_2)\right) \\
& = \deg (K_{ B_1^\nu} +K_{B_2^\nu})+ \deg \left(\frac 1 2 \left(\Diff_{\bar B_1^\nu}(\bar A_1)+R\right)+\Diff_{\bar B_2^\nu}(\bar A_2)\right)\notag\\
& = \deg K_{ B^\nu} +  \deg \left(\frac 1 2 \left(\Diff_{\bar B_1^\nu}(\bar A_1)+R\right)+\Diff_{\bar B_2^\nu}(\bar A_2)\right)\notag
\end{align}

Let $p$ be a point in $B$ such that  $\mu_p(D\cup\Delta)\geq 2$. Then we can decompose the multiplicity $\mu_p(B) = \mu_p(B_1) + \mu_p(B_2)$. Taking the degree of the maps into account we have
\begin{equation}\label{eq: decompose2}
\begin{split}
&\mu_p(B_1) = \frac 1 2\#(\nu_1\circ\pi_1)^{-1}(p) +\frac 1 2\#\left((\nu_1\circ\pi_1)^{-1}(p)\cap R\right),\\
&\mu_p(B_2) = \#(\nu_2\circ\pi_2)^{-1}(p),
\end{split}
\end{equation} 
where $\nu_i\colon B_i^\nu\rightarrow B_i$, $i=1,2$, are the normalisations and  $R$ is the set of ramification points of $\pi_1$.

The different $\Diff_{\bar B^\nu}(\bar A) = \Diff_{\bar B_1^\nu}(\bar A_1) + \Diff_{\bar B_2^\nu}(\bar A_2) $ is effective and has coefficient 1 over every singular point of $D\cup\Delta$ which lies in $B$ (Lemma~\ref{lem: multinode}). Combining this with  \eqref{eq: decompose} and \eqref{eq: decompose2}, we have
\begin{align*}
  \deg (K_X+\Delta)\restr{B} &\geq \deg K_{B^\nu} + \sum_{p\in\sing {(D\cup \Delta)}}  (\mu_p(B_1) + \mu_p(B_2))\\
 &= 2p_a(B^\nu) - 2 + \sum_{p\in\sing {(D\cup \Delta)}}  \mu_p(B)\\
\intertext{and using Lemma~\ref{lem: multinode}\refenum{iii}}
& \geq 2p_a(B) - 2+s + \sum_{p\in\sing B}  \mu_p(B) -2(\mu_p(B)-1)\\
& = 2p_a(B) - 2 +s +\sum_{p\in\sing B}  2-\mu_p(B) \\
& \geq 2p_a(B) - 2 + \sum_{p\in\sing B}  2-\mu_p(B).
\end{align*}
This concludes the proof.
\end{proof}

\begin{prop}\label{prop: restriction to D}
Let  $(X, \Delta)$ be a stable log surface with global index $I$. 
\begin{enumerate}
 \item If $m\geq4$ then the line bundle $\omega_X(\Delta)^{[mI]}\restr{ D\cup\Delta}$  is base-point-free and the associated morphism is birational.
\item If $m\geq 3$ and $I\geq 2$ then the line bundle $\omega_X(\Delta)^{[mI]}\restr{ D\cup\Delta}$  is base-point-free.
\item If $m\geq 2$ and $D\cup\Delta$ is a nodal curve then the line bundle $\omega_X(\Delta)^{[mI]}\restr{ D\cup\Delta}$  is base-point-free.
\end{enumerate}
 In each case the line bundle is very ample if the  inequality for $m$ is strict.
\end{prop}
\begin{proof}
 By the curve embedding theorem~\ref{thm: curve embedding} it suffices to check that for every (Cohen--Macaulay) subcurve $B \subset D\cup\Delta$  we have $\deg mIK_X\restr B\geq 2p_a(B)$ for base-point-freeness respectively $\deg mIK_X\restr B\geq 2p_a(B)+1$ for very ampleness. We concentrate on the base-point-freeness now, the proof for very ampleness being the same. 

We start with \refenum{iii} where $B$ is a nodal curve. Recall that $I(K_X+\Delta)$ is an ample Cartier divisor and thus has degree at least one on each irreducible component of $B$. If $p_a(B)\leq1$ the inequality $\deg mIK_X\restr B\geq 2p_a(B)$ is trivially satisfied for $m\geq2$. 
  If $p_a(B)\geq 2$ then 
by Lemma~\ref{lem: deg comparison}
\[ \deg I(K_X+\Delta)\restr B\geq  \deg (K_X+\Delta)\restr B\geq  2p_a(B) - 2 + \sum_{p\in\sing B}  2-\mu_p(B)=2p_a(B) - 2\geq 2\]
and thus $\deg 2I(K_X+\Delta)\restr B\geq 2p_a(B)$.

Now suppose $B$ is singular.  We compute, using Lemma~\ref{lem: deg comparison},
\begin{align*}
\deg 2(K_X+\Delta)\restr{B^\nu}&=  \deg (K_X+\Delta)\restr{B^\nu} +\deg (K_X+\Delta)\restr{B^\nu}\\
&\geq 2p_a(B^\nu) -2 + \sum_{p\in\sing B}  \mu_p(B) + 2p_a(B)-2   +\sum_{p\in\sing B} (2- \mu_p(B))\\
& \geq 2p_a(B)+2p_a(B^\nu)-4 + 2\#\{p\in B\,|\,\mu_p(B)\geq 2\}\\
&\geq 2p_a(B) + 2(p_a(B^\nu) -1) \hspace{1cm}\text{(because $B$ is singular)} 
\end{align*}
If $B$ has $a$ irreducible components then $p_a(B^\nu) -1=-\chi(\ko_{B^\nu})$ is at least $-a$ and thus we continue the computation to get 
\[\deg 2(K_X+\Delta)\restr{B^\nu}\geq 2p_a(B) -2a  \geq  2p_a(B) -2 \deg I(K_X+\Delta)\restr{B^\nu},\]
where the second inequality follows because  the degree of the line bundle $I(K_X+\Delta)$ is at least 1 on each irreducible component of $B^\nu$.

Therefore we have $\deg (2I + 2) (K_X+\Delta)\restr{B^\nu} \geq 2p_a(B)$. Consequently if $m\geq \frac{2 +2I}{I}$ then $\deg mI(K_X+\Delta)\restr{B}\geq 2p_a(B)$. This proves the  base-point-freeness in case \refenum{i} and \refenum{ii}.

It remains to prove that $\omega_X(\Delta)^{[mI]}\restr{ D\cup\Delta}$ gives a birational map for $m\geq4$. Since this map is very ample if $I\geq2$ or $m>4$ by the above, we may assume that $I=1$ and $m=4$.

Let $p,q$ be two smooth points of $D\cup \Delta$ and let $\ki=\ko_{D\cup\Delta}(-p-q)\subset \ko_{D\cup\Delta}$ be the corresponding (invertible) ideal sheaf. To show that $\omega_X(\Delta)^{[4]}$ separates $p$ and $q$  is suffices to prove $H^1({D\cup\Delta},\omega_X(\Delta)^{[4]}\tensor \ki)=0$.
Assume on the contrary $H^1({D\cup\Delta},\omega_X(\Delta)^{[4]}\tensor \ki)\neq0$. Then, by Lemma~\ref{lem: H^1 on curve}, there is non-empty subcurve $B\subset {D\cup\Delta}$ such that 
\begin{equation}\label{eq: subcurve of D}
\chi(B,\omega_X(\Delta)^{[4]}\tensor \ki\restr B)\leq \chi(B,\omega_B)
\end{equation}
with equality if and only if $\omega_X(\Delta)^{[4]}\tensor \ki\restr B\cong\omega_B$. 
On the other hand, we have shown above that
$\deg \omega_X(\Delta)^{[4]}\tensor \ki\restr B \geq 2p_a(B) -2$ with strict inequality if $B$ is nodal. Hence $B$ is not a nodal curve, \eqref{eq: subcurve of D} is in fact an equality and $\omega_B\isom \omega_X(\Delta)^{[4]}\tensor \ki\restr B$ is a line bundle. But the dualising sheaf a curve with a $\mu$-multi-node with $\mu\geq3$ is not a line bundle by \cite[Aside 5.9.3]{KollarSMMP}\,---\,a contradiction. 
\end{proof}

\begin{rem}
Examining  the proof closely the bounds can be improved under additional assumptions, for example if there are no rational irreducible components of ${D\cup\Delta}$ on which $\omega_X$ has small degree.

Numerically, the worst case occurs if $X$ is Gorenstein, the non-normal locus $D$ has just one singular point, and every irreducible component $B$ of $D$ is rational with a 3-multi-nodal point such that $\deg(\omega_X\restr B) = 1$. In this case, denoting by $k$ the number of irreducible components of $D$, we have $p_a(D) =2k$ and hence $\deg(\omega_X^{\otimes3}\restr D) =2p_a(D) - k$. Such examples can be constructed explicitly, see \cite{liu-rollenske13}. We will analyse the simplest such curve, a rational curve with a single 3-multi-node in Example \ref{ex: 3-multi-nodal rational curve}. We will see that the conditions of Proposition \ref{prop: restriction to D} are not necessary for base-point-freeness on such a curve but the bound for very ampleness is sharp: $\omega_X(\Delta)^{[4I]}\restr{ D\cup\Delta}$ need not be an embedding.
\end{rem}

The following technical result will be used later on.
\begin{cor}\label{cor: codim 3}
 Let $(X, \Delta) $ be a stable log surface with
global index $I$. Assume $m\geq 3$ if $D\cup\Delta$ is a nodal curve and $m\geq4$ otherwise.

Then for each irreducible component $B\subset D\cup \Delta$ the image of the restriction map 
\[ \rho_B\colon H^0(D\cup \Delta, \omega_X(\Delta)^{[mI]}\restr{D\cup\Delta}) \to H^0(B, \omega_X(\Delta)^{[mI]}\restr B)\]
has dimension at least $3$.
\end{cor}
We will see in Example \ref{ex: 3-multi-nodal rational curve} that the bound cannot be improved if $D\cup \Delta$ is not a nodal curve: it is possible that $h^0(B, \omega_X(\Delta)^{[mI]}\restr B)=2$ for $m=3$.
\begin{proof}
Let $B$  be an irreducible component of $D$ and  $n$ the dimension of the image of the restriction map $\rho_B$. By Proposition~\ref{prop: restriction to D} the map induced by $\omega_X(\Delta)^{[mI]}\restr {D\cup\Delta}$ is a birational morphism if $m\geq 4$ and an embedding if $m=3$ and $D\cup \Delta $ is nodal. Thus $\im\rho_B$ defines a birational morphism $\phi\colon B \to\phi(B)\subset  \IP^{n-1}$ and $\phi(B)$ is a non-degenerate curve of degree $\deg \omega_X(\Delta)^{[mI]}\restr B\geq 3$.  Therefore $\dim_\IC \im\rho_B\geq 3$ as claimed. 
\end{proof}

\subsection{Proof of Theorem~\ref{thm: base-point-free}}
Let $m\geq 4$, or $m\geq 3$ if one of the conditions in Theorem~\ref{thm: base-point-free}\refenum{ii} holds. Note that b) implies c) by the classification of log-canonical singularities with reduced boundary.

The restriction map $\gamma\colon H^0(X,\omega_X(\Delta)^{[mI]})\rightarrow H^0(D\cup\Delta,\omega_X(\Delta)^{[mI]}\restr{D\cup\Delta})$ is surjective by the vanishing in Corollary~\ref{cor: vanishing2}. By Proposition~\ref{prop: restriction to D} and the surjectivity of $\gamma$, we know that $\omega_X(\Delta)^{[mI]}$  has no base-points on $D\cup \Delta$ under our assumptions. 

Combining this with Corollary~\ref{cor: bpfoutsideD} we conclude that $\omega_X(\Delta)^{[mI]}$ has no base-points for $m\geq 4$ and for  $m\geq 3$ under the conditions given in Theorem~\ref{thm: base-point-free}\refenum{ii}. 

\begin{rem}
 According to the proof of Theorem~\ref{thm: base-point-free},  $\phi_{mI}$ ($m\geq 4$) already separates the points on different irreducible components of $X$, that is, the image has the same number of irreducible components; in addition no normal point of $X$ is mapped to the image of $D\cup \Delta$. 
\end{rem}

\section{Pluri-log-canonical embeddings}
In this section we prove our results on pluri-log-canonical embeddings. For a general stable log surface we can prove the following:

\begin{theo}\label{thm: main general}
Let $(X, \Delta)$ be a connected stable log surface of global index $I$.
\begin{enumerate}
\item The line bundle $\omega_X(\Delta)^{[mI]}$  is very ample for $m\geq 8$.
\item The line bundle $\omega_X(\Delta)^{[mI]}$  defines a birational morphism for $m\geq6$.
\item The line bundle $\omega_X(\Delta)^{[mI]}$  is very ample for $m\geq 6$ if $I\geq2$.
\end{enumerate}
\end{theo}
Under further assumptions on singularities and invariants we can improve these bounds. 
\begin{thm}\label{thm: main special}
 Let $(X, \Delta)$ be a connected stable log surface of global index $I$.
\begin{enumerate}
 \item The line bundle $\omega_X(\Delta)^{[mI]}$  is very ample  for $m\geq 7$ if one of the following holds:
\begin{enumerate}
\item There is no irreducible component $\bar X_i$ of the normalisation such that $\left(\pi^*(K_X+\Delta)\restr {\bar X_i}\right)^2=1$, and the union of $\Delta$ and the non-normal locus is a nodal curve.
\item  $X$ is normal and not $(K_X+\Delta)^2=1$.
\end{enumerate}
\item The line bundle $\omega_X(\Delta)^{[mI]}$  is very ample  for $m\geq6$ if the normalisation $\bar X$ is smooth along the conductor divisor and has at most canonical singularities elsewhere. 
 \item The line bundle $\omega_X(\Delta)^{[mI]}$  is very ample  for $m\geq5$ if $D\cup\Delta$ is a nodal curve,  $\bar X$ is smooth along the conductor divisor, and $X\setminus D$ has at most canonical singularities.
In particular these conditions are satisfied, if $(X,\Delta)$ has semi-canonical singularities.
\end{enumerate}
\end{thm}
We cannot prove and do not believe all bounds given in Theorems~\ref{thm: main general} and \ref{thm: main special} to be sharp; we will discuss some evidence for this in Remark \ref{rem: not sharp}.

\begin{rem}
One should resist the temptation to believe that $X$ is Gorenstein if the normalisation is smooth: for example if we take $\bar D \subset \bar X$ to be the coordinate axes in the plane and pinch both lines via the restriction of $p\mapsto -p$ then the resulting slc surface has index 2. A local trivialising  section of $\omega_{\bar X}(\bar D)$ is $\frac{dx\wedge dy}{xy}$ and the residue along the $x$-axis $-dx/x$ is invariant under the involution and thus does not descend to $X$ by Proposition \ref{prop: sections}.

However, in this case the index is always at most 2.
\end{rem}

\subsection{Outline of the proof and preliminary results}\label{sect: prelim for embedding}
The strategy of our proof is classical: for every subscheme $\xi$ of length two, find an appropriate pluri-log-canonical curve $C$ containing it, and then prove that $mI(K_X+\Delta)$ embeds this curve, and hence $\xi$, by the numerical criterion of Theorem \ref{thm: curve embedding}.
There are two obstacles that restrict the choice of a curve $C$ that we can handle: 
\begin{itemize}
 \item If $C$ has a common irreducible component with $D\cup\Delta$, i.e., $C$ is not log-well-behaved (Definition~\ref{def: wb}), then it is not under control when we pull it back to the  normalisation to compute intersection numbers and the adjunction. 
\item If $(\bar X, \bar\Delta+\bar D)$ has worse than canonical singularities then some of the connectedness properties of ample curves are needed to counter-weigh the contributions from singularities, so in this case we can only effectively handle reduced curves $C$. 
\end{itemize}
We start with some preliminary considerations how to construct log-well-behaved  curves and how to get around the failure of the adjunction formula. At the end we include a version of the connectedness lemma for ample Cartier divisors on normal surfaces.

\subsubsection{Construction of log-well-behaved curves}

\begin{lem}\label{lem: curve}
Let $(X, \Delta)$ be a stable log surface and $\xi\subset X$  a subscheme of length 2.
\begin{enumerate}
 \item If $\omega_X(\Delta)^{[mI]}$ has no base-points but $\phi_{mI}\restr \xi$ is not an embedding then there exists a log-well-behaved reduced curve $C\in|mI(K_X+\Delta)|$  containing $\xi$.
\item 
Assume $m\geq 3 $ and $m\geq 4$ if $D\cup\Delta$ is not a nodal curve.
Then there exists a log-well-behaved curve $C$  in $|mI(K_X+\Delta)|$ containing $\xi$.

In addition, if $|mI(K_X+\Delta)|$ has no base-points then  $C$  can be chosen to be reduced unless $\phi_{mI}\restr\xi$ is an embedding and the line spanned by $\phi_{mI}(\xi)$ is contained in the branch locus of $\phi_{mI}$.
\end{enumerate}
\end{lem}
\begin{proof}
Let $\xi\subset X$ be an arbitrary subscheme of length 2  and $\ki_\xi$ its ideal sheaf. Recall that if $\phi_{mI}$ is a morphism then it is automatically finite (Remark~\ref{rem: finitemorphism}). 

In case \refenum{i} let $Z\subset \IP^{h^0(\omega(\Delta)^{[mI]})-1}$ be the image of $\phi_{mI}$ and $p$ be the image of $\xi$, a reduced point. Then the preimage of every hyperplane containing
$p$ contains $\xi$ and the base locus of this linear system of hyperplanes is exactly the point $p$. Thus a general hyperplane section of $Z$ containing $p$ is reduced by Bertini and does not contain any irreducible component of the image of the non-normal locus or of the branch locus of $\phi_{mI}$. Pullback gives the required curve $C$.

We now prove \refenum{ii}. 
First assume that $D\cup\Delta$ is empty, i.e., we are on a normal surface without boundary. Then by Blache's version of Riemann--Roch for normal surfaces \cite[3.4,~3.3(c),~2.1(d)]{bla95a} and Proposition~\ref{prop: vanishing}, which applies as $m\geq 3$, we have
\begin{align*}
h^0(X, \omega_X^{[mI]})&=\chi(\ko_X(mIK_X))= \chi(\ko_X)+\frac{m(mI-1)}{2}I K_X^2\\
&\geq \chi(\ko_X)+\frac{m(mI-1)}{2}\geq \chi(\ko_X)+3\geq 4\\
\end{align*}
where in the last step we used $\chi(\ko_X)\geq 1$ from  {\cite[Theorem~2]{bla94}}.
Thus there are at least 4 sections and at least a 2-dimensional space of sections vanishing on $\xi$. 

Now assume $D\cup \Delta$ is non-empty. Consider, for an irreducible component $B$ of $D\cup\Delta$, the diagram with exact rows and columns
\[ \begin{tikzcd}
 0\dar \\
 H^0(X, \omega_X(\Delta)^{[mI]}\tensor \ki_\xi) \dar\arrow{dr}{\psi}\\
 H^0(X, \omega_X(\Delta)^{[mI]})\rar]&H^0(D, \omega_X(\Delta)^{[mI]}\restr D) \rar\dar{\rho_B} &0\\
&H^0(B, \omega_X(\Delta)^{[mI]}\restr B). 
\end{tikzcd}
\]
By Corollary~\ref{cor: codim 3} the kernel of $\rho_B$ has codimension at least 3 under our assumptions while the image of $\psi$ has codimension at most 2. Thus a general section in $ H^0(X, \omega_X(\Delta)^{[mI]}\tensor \ki_\xi)$ does not restrict to zero on any irreducible component of $D\cup \Delta$. 

Now assume in addition that $|mI(K_X+\Delta)|$ has no base-points.  Because of \refenum{i}, to get a reduced curve we only need to consider the case where $\phi_{mI}\restr\xi$ is an embedding. In that case $\phi_{mI}(\xi)$ spans a line which is the base locus of the linear system of hyperplanes  whose preimage contains $\xi$. Since a general curve $C\in|mI(K_X+\Delta)|$ that contains $\xi$ is log-well-behaved, as is shown above, we can choose a hyperplane with reduced preimage if and only if the line is not contained in the branch locus of $\phi_{mI}$.
\end{proof}

\subsubsection{Corrections to adjunction}\label{sect: adjunction}
Now our aim is to bound  the correction terms occurring in the adjunction formula for a log-well-behaved curve on a stable log surface. 

Let $(X,\Delta)$ be a stable log surface. We consider the minimal semi-resolution $f\colon Y\rightarrow X$. Let $\eta\colon \bar Y \rightarrow Y$ be the normalisation whose conductor divisor is denoted by $D_{\bar Y}$. The map $\bar Y \to X$ factors through $\bar X$ the normalisation of $X$, whose conductor divisor will now be denoted $D_{\bar X}$ (instead of $\bar D$); we get a commutative diagram
\[
\begin{tikzcd}
 \bar Y \rar{\bar f} \dar{\eta} & \bar X\dar{\pi}\\ Y\rar{f} & X.
\end{tikzcd}
\]
Let $B\subset X$ be a log-well-behaved curve. We fix some notation to formulate the first result:
\begin{itemize}
\item $B_{\bar Y}\subset \bar Y$ and  $B_{\bar X} \subset \bar X$ are the strict transforms of $B$.
\item   $\hat B_Y\subset Y$ is the  hat transform of $B$, defined in Appendix~\ref{section: hattransform}.
\item  $\hat\Gamma_{B_{\bar X}} = \hat B_{\bar Y} - B_{\bar Y}$ and  $\Gamma_{B_{\bar X}}^* = \bar f^*B_{\bar X} - B_{\bar Y}$.
\item $\Delta_{\bar X}\subset \bar X$ the strict transform of $\Delta$. 
 \item $\Lambda=\bar f^* (K_{\bar X} +D_{\bar X}) - (K_{\bar Y}+D_{\bar Y})$ is the codiscrepancy of the pair $(\bar X, D_{\bar X})$. Note that $\Lambda$ is effective because $ f\colon Y\rightarrow  X$ is the minimal semi-resolution \cite[Prop.~4.12]{ksb88}. 
\end{itemize}

In the next lemma we estimate the failure of adjunction on the singular surface $X$ in terms of data on the normalisation of the minimal semi-resolution.
\begin{lem}\label{lem: correction to adjunction}
 Let $(X, \Delta)$ be a stable log surface and let $B\subset X$ be a log-well-behaved, not necessarily reduced curve. Then with the above notation we have
\[
(K_X+\Delta+B)B\geq (K_X+B)B
\geq2p_a(B)-2 + (\Lambda - \hat\Gamma_{B_{\bar X}} + \Gamma^*_{B_{\bar X}})(\hat\Gamma_{B_{\bar X}} - \Gamma^*_{B_{\bar X}}).
\]
\end{lem}
\begin{proof}
Since $B$ is log-well-behaved, $\Delta B = \Delta_{\bar X} B_{\bar X}\geq 0$ and it suffices to prove the second inequality. By Proposition~\ref{prop: genus hat transform}, $p_a(B) \leq p_a(\hat B_{\bar Y}) + \frac{\hat B_{\bar Y}\hat D_{\bar Y}}{2}$, and hence, by adjunction on $\bar Y$, we have
\[
 2 p_a(B) -2 \leq (K_{\bar Y} + D_{\bar Y} +\hat B_{\bar Y}) \hat B_{\bar Y}. 
\]
It follows that
\begin{align*}
&(K_{X} +B)B - (2p_a(B)-2) \\
=\,&(K_{\bar X} + D_{\bar X} + B_{\bar X})B_{\bar X}  - (2p_a(B)-2) \\
\geq\, & (K_{\bar X} + D_{\bar X} +  B_{\bar X}) B_{\bar X} - (K_{\bar Y} + D_{\bar Y} + \hat B_{\bar Y}) \hat B_{\bar Y}\\ 
  = \,& (K_{\bar Y} + D_{\bar Y} + \Lambda + B_{\bar Y} + \Gamma^*_{B_{\bar X}}) B_{\bar Y} - (K_{\bar Y} + D_{\bar Y} + B_{\bar Y} + \hat\Gamma_{B_{\bar X}})(B_{\bar Y}+\hat\Gamma_{B_{\bar X}})\\
 = \, &(\Lambda+\Gamma^*_{B_{\bar X}} -\hat \Gamma_{B_{\bar X}})B_{\bar Y}- (K_{\bar Y}+ D_{\bar Y} + B_{\bar Y} + \hat\Gamma_{B_{\bar X}})\hat\Gamma_{B_{\bar X}}\\
\intertext{and since $B_{\bar Y}\bar E_i = -\Gamma^*_{B_{\bar X}}\bar E_i$ and $(K_{\bar Y}+ D_{\bar Y} + \Lambda)\bar E_i=\pi^*K_X\bar E_i=0$ for any $i$}
 = \, &(\Lambda+\Gamma^*_{B_{\bar X}} -\hat \Gamma_{B_{\bar X}})(-\Gamma^*_{B_{\bar X}})- (- \Lambda -\Gamma^*_{B_{\bar X}} + \hat\Gamma_{B_{\bar X}})\hat\Gamma_{B_{\bar X}}\\
   =\, &(\Lambda - \hat\Gamma_{B_{\bar X}} + \Gamma^*_{B_{\bar X}})(\hat\Gamma_{B_{\bar X}} - \Gamma^*_{B_{\bar X}}),
\end{align*}
Bringing $2p_a(B)-2$ to the other side gives the second inequality.\end{proof}

\begin{lem}\label{lem: bound}
Let $(X, \Delta)$ be a stable log surface of global index $I$, $m\geq 1$ and let  $C\in|mI(K_X+\Delta)|$ be a log-well-behaved reduced curve. Then for every subcurve $B\subset C$ we have, in the notation introduced in Section~\ref{sect: adjunction},
\[ (mI+1)(K_X+\Delta)B = (mI+1)(K_{\bar X} + D_{\bar X} +\Delta_{\bar X}) B_{\bar X} \geq 2p_a(B) -2.\]
\end{lem}
\begin{proof}
The first equality is clear and we only prove the second. We decompose $C = A + B$ as a Weil divisor and let  $A_{\bar X}$, $A_{\bar Y}$, $C_{\bar X}$, and  $C_{\bar Y}$ be the strict transform of $A$ and $C$ in $\bar X$ resp.\ $\bar Y$.

We have 
\begin{align*}
 mI(K_{\bar X} + D_{\bar X} + \Delta_{\bar X}) B_{\bar X} -B_{\bar X}^2 &=(C_{\bar X}-B_{\bar X})B_{\bar X}\notag\\
&= A_{\bar X} B_{\bar X}\notag\\
                                &= (\bar f^* A_{\bar X})  (\bar f^*B_{\bar X})\notag\\
                                &= (A_{\bar Y} + \Gamma_{A_{\bar X}}^*) B_{\bar Y}\\
                                &\geq \Gamma_{A_{\bar X}}^* B_{\bar Y}\hspace{1cm} \text{ (since $C_{\bar Y}$ is reduced, $A_{\bar Y}B_{\bar Y}\geq 0$)}\notag\\
                                &= - (\Gamma_{C_{\bar X}}^*  - \Gamma_{B_{\bar X}}^*) \Gamma_{B_{\bar X}}^*\notag
\end{align*}
where in the last line we use that  $\bar E(B_{\bar Y}+\Gamma_{B_{\bar X}}) = 0$ for every $\bar f$-exceptional curve $\bar E$.

Adding this to the equation resulting from Lemma~\ref{lem: correction to adjunction} we get
\begin{multline}\label{eq: compare}
 (mI+1)(K_{\bar X} + D_{\bar X} + \Delta_{\bar X}) B_{\bar X} \\ \geq 2p_a(B)-2 -(\hat\Gamma_{B_{\bar X}} - \Gamma_{B_{\bar X}}^*) ( \hat\Gamma_{B_{\bar X}}- \Lambda) - \Gamma_{B_{\bar X}}^*(\Gamma_{C_{\bar X}}^*-\hat\Gamma_{B_{\bar X}}).
\end{multline}
By the definition of hat transform (Definition~\ref{defin: hattransform})  the intersection numbers of $\hat\Gamma_{B_{\bar X}} - \Gamma_{B_{\bar X}}^*$ and $\Gamma_{B_{\bar X}}^*$ with any exceptional divisor of $\bar f$ are non-positive. On the other hand $\Gamma_{C_{\bar X}}^*-\hat\Gamma_{B_{\bar X}}\geq 0$ by  Lemma~\ref{lem: easy properties of hattransform}\refenum{iii}. Also  $\hat\Gamma_{B_{\bar X}}- \Lambda$ has non-negative coefficients at every  exceptional divisors mapped to $B_{\bar X}$ because at each of those $\hat\Gamma_{B_{\bar X}}$ has coefficients at least $ 1$  while the coefficients of $\Lambda$ are at most $1$ for the log-canonical pair $(\bar X,D_{\bar X})$.  So 
\[ 
 -(\hat\Gamma_{B_{\bar X}} - \Gamma_{B_{\bar X}}^*) ( \hat\Gamma_{B_{\bar X}}- \Lambda) - \Gamma_{B_{\bar X}}^*(\Gamma_{C_{\bar X}}^*-\hat\Gamma_{B_{\bar X}})\geq 0\]
 and the claim follows from \eqref{eq: compare}.
\end{proof}

\subsubsection{Connectedness of ample Cartier divisors on normal surfaces}\label{section: connectedness}
  This section provides a connectedness result about ample Cartier divisors on normal surfaces.
\begin{lem}\label{lem: connect}
 Let $X$ be a projective normal surface and $M$ an ample Cartier divisor on $X$.
 Let $n\in\IN_{\geq 2}$ and $ C\in|n M|$. Suppose $ C= C_1+ C_2$ is a decomposition into two (non-empty) curves. Then 
  \[ C_1  C_2 \geq n -\frac 1{M^2}\geq n-1, \]
   and $C_1  C_2 = n-1$ if and only if $M^2=1$ and one of the $C_i$ is numerically equivalent to $M$.
  \end{lem}
  \begin{proof}
  We can numerically write (\cite[\S 4, Lem.~1]{bombieri73})
  \[
    C_1 \equiv a M + \epsilon,\  C_2 \equiv (n-a)M - \epsilon, 
  \]
  where $a=\frac{M C_1}{M^2}$ and $M\epsilon=0$. 

If $f\colon Y\to X$ is a resolution then, by \cite[p.~878]{sak84}, the Picard lattice of $Y$ contains the subspace of $f$-exceptional curves as a direct summand on which the intersection form is negative definite. Thus the Hodge index theorem on $Y$ implies that  the intersection form  on $X$ has signature $(1,k)$ for some $k\geq 0$. Hence $- \epsilon^2 \geq 0$, with equality if and only if $\epsilon\equiv 0$. 
  
  Since $M$ is an ample Cartier divisor, we have $M C_i >0$ for $i=1,2$, and both of the intersection numbers are integers. Therefore
   \[
    a=\frac{M C_1}{M^2}\geq\frac{1}{M^2}
   \]
   and also
  \[
  \frac{1}{M^2} \leq\frac{M C_1}{M^2}=a=n- \frac{M C_2}{M^2}\leq n - \frac{1}{M^2}.
  \]
The expression   $a(n-a)M^2$, considered as a  quadratic function in $a$,  attains its minimum for the smallest (or biggest) possible value of $a$ and  thus
  \begin{equation}\label{equation: connectedness}
     C_1 C_2=a(n-a)M^2 - \epsilon^2 \geq n - \frac{1}{M^2} - \epsilon^2\geq n - \frac{1}{M^2} \geq n - 1. 
  \end{equation}
 The inequalities in \eqref{equation: connectedness} are all equalities if and only if $M^2=1$, $\epsilon\equiv 0$, and  $a=aM^2=1$ or $a=aM^2=n-1$. This is possible if and only if one of the curves is numerically equivalent to $M$.  
  \end{proof}

\subsection{Proof of Theorem \ref{thm: main general}\refenum{i}}
Let  $\xi\subset X$ be a subscheme  of length 2. By Theorem \ref{thm: base-point-free} the 4-canonical map $\phi_{4I}$ is a morphism. If $\phi_{4I}\restr\xi$ is an embedding then $\phi_{mI}\restr\xi$ is also an embedding for $m\geq 8$, because $|(m-4)I(K_X+\Delta)|$ is base-point-free, again by Theorem \ref{thm: base-point-free}. If $\phi_{4I}\restr\xi$ is not an embedding then by Lemma~\ref{lem: curve} there exists a log-well-behaved reduced curve $C$ containing $\xi$.
Proposition~\ref{prop: vanishing} yields a surjection of global sections
\[
 H^0(X,\omega_X(\Delta)^{[mI]}) \twoheadrightarrow H^0(C,\omega_X(\Delta)^{[mI]}\restr{C}) \text{ for } m\geq 6.
\]
Therefore to show that $\phi_{mI}$ ($m\geq 8$) is an embedding, it suffices to show that $\omega_X(\Delta)^{[mI]}\restr{C}$ defines an embedding for any subscheme $\xi$ of length two that is contracted by $\phi_{4I}$. By Theorem~\ref{thm: curve embedding} it suffices to show that, for any subcurve $B\subset C$, we have $8I(K_X+\Delta) B \geq 2p_a(B)+1$.

By Lemma~\ref{lem: bound}, applied to $B$, we have $(4I+1)(K_X+\Delta)B \geq 2p_a(B) -2$. Since $I(K_X+\Delta)$ is an ample Cartier divisor and thus has degree at least 1 on $B$, we obtain 
\[
8I(K_X+\Delta) B\geq   2p_a(B) -2 + (4I-1)(K_{\bar X} + D_{\bar X} +\Delta_{\bar X})B_{\bar X}  \geq 2p_a(B)+1,
\]
which concludes the proof.

\begin{rem}\label{rem: not sharp}
Employing a trick used below in the proof of Theorem~\ref{thm: main general}\refenum{iii} one could get a better bound of $7I$ for those $\xi \in X$  that are not embedded by $\phi_{4I}$. This does not allow us to conclude that $\phi_{7I}$ is very ample in general: let $\xi$ be a subscheme of length two  such that $\phi_{4I}\restr{\xi}$ is an embedding. Then $\phi_{7I}\restr \xi$ is an embedding at $\xi$ unless $\xi$ is supported on a base-point of $3IK_X$. 

However, in the latter case we do not know how to find a log-well-behaved reduced curve in $|3I(K_X+\Delta)|$ or $|4I(K_X+\Delta)|$ containing $\xi$. This seems to be an artefact of our method and we are led to believe that the bound in  Theorem~\ref{thm: main general}\refenum{i} is not sharp. 
\end{rem}

\subsection{Proof of Theorem~\ref{thm: main general}\refenum{iii} and  Theorem~\ref{thm: main special}\refenum{i}} 
Let  $\xi\subset X$ be a subscheme  of length 2 and assume we are in the case of Theorem~\ref{thm: main general}\refenum{iii} or  Theorem~\ref{thm: main special}\refenum{i}.
Then by Theorem~\ref{thm: base-point-free} the tri-canonical map $\phi_{3I}$ is a morphism. If $\phi_{3I}\restr\xi$ is an embedding then $\phi_{mI}\restr\xi$ is also an embedding for $m\geq 6$, because $|(m-3)I(K_X+\Delta)|$ is base-point-free, again by Theorem~\ref{thm: base-point-free}. If $\phi_{3I}\restr\xi$ is not an embedding then by Lemma~\ref{lem: curve} there exists a log-well-behaved reduced curve $C$ containing $\xi$.

Proposition~\ref{prop: vanishing} yields a surjection of global sections
\[
 H^0(X,\omega_X(\Delta)^{[mI]}) \twoheadrightarrow H^0(C,\omega_X(\Delta)^{[mI]}\restr{C}) \text{ for } m\geq 5.
\]
Therefore to show that $\phi_{mI}$ ($m\geq 6$) is an embedding, it suffices to show that $\omega_X(\Delta)^{[mI]}\restr{C}$ defines an embedding. By Theorem~\ref{thm: curve embedding} it suffices to show that, for any subcurve $B\subset C$, we have $mI(K_X+\Delta) B \geq 2p_a(B)+1$. As $m\geq 5$, this is trivial if $p_a(B)\leq 2$, so we assume $p_a(B)\geq 3$.

By Lemma~\ref{lem: bound}, applied to $B$, we have $(3I+1)(K_X+\Delta)B \geq 2p_a(B) -2$ and thus
\begin{align*}
   6I(K_X+\Delta)B &\geq 2p_a(B) - 2 + \frac{3I-1}{3I+ 1}(2p_a(B)-2)\\
& \geq 2p_a(B) -2 +\frac{4(3I-1)}{3I+ 1} \qquad(\text{since }p_a(B)\geq 3)\\
&\geq 
\begin{cases} 2p_a(B) + \frac{6}{7} & I\geq 2\\ 2p_a(B)& I=1 \\
\end{cases}.
\end{align*}
Since $6I(K_X+\Delta)B$ is an integer, we have the required $6I(K_X+\Delta)B\geq 2p_a(B)+1$ if $I\geq 2$. If $I=1$ then $7I(K_X+\Delta)B\geq 2p_a(B)+1$ because $(K_X+\Delta)B\geq 1$.\qed

\subsection{Proof of Theorem~\ref{thm: main special}\refenum{ii}}\label{sect: pf of}
For any subscheme $\xi\subset X$ of length two we have a log-well-behaved curve $C\in |4K_X|$ (not necessarily reduced) containing $\xi$ by Lemma~\ref{lem: curve}\refenum{ii}.

Proposition~\ref{prop: vanishing} yields a surjection of global sections
\[
 H^0(X,\omega_X(\Delta)^{[mI]}) \twoheadrightarrow H^0(C,\omega_X(\Delta)^{[mI]}\restr{C})) \text{ for } m\geq 6.
\]
Therefore to show that $\phi_{mI}$ ($m\geq 6$) is an embedding, it suffices to show that $\omega_X(\Delta)^{[mI]}\restr{C}$ defines an embedding for any subscheme $\xi$ of length two. By Theorem~\ref{thm: curve embedding} it suffices to show that, for any subcurve $B\subset C$, we have $6I(K_X+\Delta) B \geq 2p_a(B)+1$.

We continue to use the notation from Section~\ref{sect: adjunction}. By assumption, $\bar X$ is smooth along $D_{\bar X}$ and has canonical singularities elsewhere. Thus  $\Lambda=\bar f^* (K_{\bar X} +D_{\bar X}) - (K_{\bar Y}+D_{\bar Y})$ is supported  on the preimages of the nodes of $D_{\bar X}$. 
On the other hand the divisor $B_{\bar X}$, the strict transform of $B$ in the normalisation, is Cartier in a neighbourhood of $D_{\bar X}$. The hat transform was defined in terms of intersection numbers, which are defined via the normalisation, and thus $\hat\Gamma_{B_{\bar X}} - \Gamma^*_{B_{\bar X}}$ is trivial on those exceptional divisors mapping to the nodes of $D_{\bar X}$. 
Therefore $\Lambda$ and $\hat\Gamma_{B_{\bar X}} - \Gamma^*_{B_{\bar X}}$ have disjoint support on $\bar Y$ and the intersection number $\Lambda(\hat\Gamma_{B_{\bar X}} - \Gamma^*_{B_{\bar X}})=0$.

Lemma~\ref{lem: correction to adjunction} now implies
\begin{align}
(K_X+\Delta+B)B &\geq  2p_a(B)-2 + (\Lambda - \hat\Gamma_{B_{\bar X}} + \Gamma^*_{B_{\bar X}})(\hat\Gamma_{B_{\bar X}} - \Gamma^*_{B_{\bar X}})\notag\\
&\geq  2p_a(B)-2 -(\hat\Gamma_{B_{\bar X}} - \Gamma^*_{B_{\bar X}})^2\label{eq: compare5}\\
&\geq  2p_a(B)-2\notag
\end{align}
because the intersection form is negative definite on the exceptional divisors of $\bar f$.

If $B=C\in |4I(K_X+\Delta)|$ then $B$ is a well-behaved Carter divisor and adjunction gives 
\begin{align*}
 6I(K_X+\Delta) B & =  (K_X + B ) B + \Delta B+(2I-1)(K_X +\Delta)B\\
&\geq 2p_a(B) -2 + 4(2I-1)I(K_X+\Delta)^2 > 2p_a(B) +1.
\end{align*}
If $B<C$ then there is at least one irreducible component $\bar X_i$ of $\bar X$ such that $$B_{\bar X_i} :=B_{\bar X}\cap \bar X_i < C_{\bar X}\cap \bar X_i=:C_{\bar X_i}.$$ Now Lemma~\ref{lem: connect} says that  $C_{\bar X_i} B_{\bar X_i} - B_{\bar X_i}^2 \geq 3$. Hence 
\[4IK_X B - B_{\bar X}^2 = C_{\bar X} B_{\bar X} - B_{\bar X}^2 \geq C_{\bar X_i} B_{\bar X_i} - B_{\bar X_i}^2 \geq 3.\] Combining with \eqref{eq: compare5}, we have 
\[
 6I(K_X+\Delta) B \geq 2p_a(B) +1.
\]
As before   $\phi_{6I}$ embeds $C$ by Theorem~\ref{thm: curve embedding}. This implies that $\phi_{6I}$ embeds $\xi$. Since $\xi$ is an arbitrary subscheme of length two, $\phi_{6I}$ embeds $X$.\qed

\subsection{Proof of Theorem~\ref{thm: main special}\refenum{iii}}
The proof is exactly the same as for the previous case with the twist that,  under our assumptions,  we can  choose the curve $C$ to be contained in $|3IK_X|$ by Lemma~\ref{lem: curve}. Even though we get weaker connectedness from Lemma~\ref{lem: connect}, the numerical criterion is still satisfied for $m\geq 5$.\qed

 \subsection{Proof of Theorem~\ref{thm: main general}\refenum{ii}}
Let $S$ be the finite subset of $X$ defined as the union of 
\[
\{P\in D\,|\, \pi^{-1}(P) \text{ contains a singular point of }\bar X\}  
\]
and
\[\{P\in X\setminus D\,|\, P \text{ is worse than a canonical singularity}\}.\]
Let $\xi$ be the length 2 subscheme consisting of two general points in $X$. Then there exists a log-well-behaved curve $C\in |4IK_X|$ containing $\xi$ and choosing $C$ general we can assume that $C$ does not intersect $S$.  Repeating the argument from the proof of Theorem~\ref{thm: main special}\refenum{ii} in Section~\ref{sect: pf of} we conclude that $\phi_{6I}\restr C$ is an embedding. Thus $\phi_{6I}$ is a morphism that separates every two general points in $X$, hence is birational.\qed

\section{The log-canonical ring}\label{sect: canonical ring}
 For a log surface $(X, \Delta)$ the log-canonical ring is 
\[R(X, K_X+\Delta) = \bigoplus_{m\geq 0} H^0(X, \omega(\Delta)^{[m]}).\]
In this section we study the implications of our results so far for  this ring.
\begin{theo}\label{thm: ring}
 Let $(X, \Delta)$ be a stable surface of index $I$.
Assume that $\omega_X(\Delta)^{[aI]}$ is generated by global sections. Then for $k\geq 2+2aI$ the multiplication maps 
\[ H^0(X, \omega_X(\Delta)^{[k]})\tensor H^0(X, \omega_X(\Delta)^{[aI]}) \to H^0(X, \omega_X(\Delta)^{[k+aI]})\]
are surjective and the log-canonical ring is generated in degree at most $3aI+1$.

 In particular,
\begin{enumerate}
\item  $R(X, K_X+\Delta)$ is generated in degree at most $12I+1$, and 
\item   $R(X, K_X+\Delta)$ is generated in degree at most $9I+1$  if one of the following holds:
\begin{enumerate}
\item $I\geq 2$, 
\item There is no irreducible component $\bar X_i$ of the normalisation such that $\left(\pi^*(K_X+\Delta)\restr {\bar X_i}\right)^2=1$, and the non-normal locus is a nodal curve.
\item $X$ is normal and we do not have $I=(K_X+\Delta)^2=1$.
\end{enumerate}
\end{enumerate}
 \end{theo}

\begin{proof}
The line bundle $M:=\omega_X(\Delta)^{[aI]}$ is generated by global sections and ample.  By Proposition~\ref{prop: vanishing} we have
\[H^i(X, \omega_X(\Delta)^{[k]}(-iM))=H^i(X,\omega_X(\Delta)^{[k-iaI]})=0\]
for $i>0$ and $k\geq 2+2aI$; we say $\omega_X(\Delta)^{[k]}$ is $0$-regular.
Thus the multiplication map
is surjective for $k\geq 2+2aI$ by Mumford's Lemma \cite[Thm.~1.8.5]{LazarsfeldI}, which is also valid for reducible varieties. The statement on the generators of the ring follows.

For the second part note that we can always choose $a=4$  and $a=3$ under the stronger assumptions given in \refenum{ii} by  Theorem ~\ref{thm: base-point-free}.
\end{proof}
\begin{rem}
One can also deduce from \cite[Thm.~3]{Mumford70} that the line bundles $\omega_X(\Delta)^{[12I]}$ in case \refenum{i} respectively $\omega_X(\Delta)^{[9I]}$ in case \refenum{ii}  satisfies property $N_1$, that is, the image of $\phi_{12I}$ respectively $\phi_{9I}$ is projectively normal and cut out by quadrics.
\end{rem}

\section{Examples}\label{sect: examples}
In this section we construct some examples of stable surfaces and analyse line bundles on a rational curve with a single 3-multi-node.

We concentrate on examples strictly related to the topic of this article; for further constructions and observations we refer to \cite{liu-rollenske13}.

\begin{exam}[Very ampleness of $K_{\bar X} +\bar D$ does not descend]\label{ex: descend}
 Let  $\bar D$ be a smooth plane quartic curve invariant under the involution of $\tau(x,y,z)=(-x,-y,z)$ on $\IP^2$; to be concrete set $\bar D = \{f=x^4+y^4+z^4=0\}$.
Then let $X$ be the (semi-smooth) stable surface corresponding to the triple $(\IP^2, \bar D, \tau\restr{\bar D})$, that is, we glue $\bar D$ to itself via $\tau$.
The quotient $D=\bar D/\tau$ is an elliptic curve and thus by Proposition~\ref{prop: invariants} the invariants of $X$ are $K_X^2=1$ and $\chi(\ko_X) = 3$.

We will now study the canonical ring $R= \bigoplus_k H^0(X, \omega_X^{\tensor k})$ of $X$ and show that while $\pi^*\omega_X^{\tensor k} \isom \ko_{\IP^2}(k)$ is very ample for $k\geq1$ the line bundle $\omega_X^{\tensor k}$ is very ample only for $k\geq 5$.

Consider the residue sequence $0\to \omega_{\IP^2}({\bar D})^{\tensor k}(-{\bar D})\to \omega_{\IP^2}({\bar D})^{\tensor k} \to \omega_{\bar D}^{\tensor k}\to 0$ which gives
\[0\to H^0(\IP^2, \ko(k-4))\to H^0(\IP^2,\omega_{\bar X}(\bar D)^{\tensor k} )\overset{R}\longrightarrow H^0(\bar D, \omega_{\bar D}^{\tensor k})\to 0.\]

It turns out that, if we identify $H^0(\IP^2,\omega_{\bar X}(\bar D)^{\tensor k} )$ with elements of degree $k$ in $S = \IC[x,y,z]$ then the residue of a  section is (anti)-invariant if and only if its residue is zero or the section is (anti)-invariant under the induced action of $\tau$ on the polynomial ring $S$. Thus, by Proposition~\ref{prop: sections},  $R_k=f\cdot S_{k-4}+ S_k^{\pm \tau}$ where $S_k^{\pm \tau}$ are the invariant or anti-invariant polynomials of degree $k$ according to the parity of $k$.

Writing out the first degrees explicitly it is easy to see that as a subring we have
\[ R = \IC[x,y, z^2, zf]\subset S,\]
and $X$ is the hypersurface in $\IP(1,1,2,5)$ given by the equation $w_4^2 -w_3(w_1^4+w_2^4+w_3^2)^2=0$.  In particular, as long as $k\leq 4$ the $k$-canonical map factors over the quotient $\IP^2/\tau=\IP(1,1,2)$ and it is very ample for $k\geq 5$. So while $\omega_{\bar X}(\bar D)^{\tensor k}$ is very ample on $\bar X = \IP^2$ for every $k\geq 1$ this very ampleness does not descend to $X$.

Incidentally the canonical ring of a smooth surface of general type with $p_g=2$ and $K^2 =1$ is known to be of the same form \cite[VII.(7.1)]{BHPV}, so we have constructed a surface in the boundary of that irreducible component of the moduli space of smooth surfaces.

Note that this example also shows that our Ansatz to prove base-point-freeness is sharp: the space of sections of $\omega_X^{\tensor k}$ vanishing along the non-normal locus might be empty for $k<4$.
\end{exam}

\begin{exam}[Large $K_X^2$ is not enough for non-normal surfaces]\label{exam: large K^2}
For a (connected) stable surface $X$  with canonical singularities, the bi-canonical map is a morphism as soon as $K_X^2\geq 5$  and the tri-canonical map is an embedding as soon as $K_X^2\geq 6$ (see \cite{catanese87}).

 We will now construct examples of non-normal stable surfaces (Gorenstein and irreducible) with $K_X^2$ arbitrarily large such that  the bi-canonical map not a morphism, and neither the tri-canonical not the 4-canonical map is an embedding. Morally, the obstructions to being base-point-free as well as the increase of $K_X^2$ happen locally, so that they cannot affect each other.

Fix once for all an inhomogeneous coordinate $z$ on $\IP^1$ and let $\tau_0(z)=-z$. On $\bar X = \IP^1\times \IP^1$ let $H_x = \IP^1\times\{x\} $ and $V_x=\{x\}\times \IP^1$ and consider for $k\geq 2$ the divisor 
\[\bar D_k = H_0+H_1+H_\infty+\sum_{j=1}^{k}(V_j+V_{-j}).\]
We specify an involution $\tau$ on the normalisation 
$\bar D_k^\nu=H_0\sqcup H_1\sqcup H_\infty\sqcup \bigsqcup_{j=1}^{k}(V_j\sqcup V_{-j})$
by 
\begin{gather*}
\tau\restr {H_0}=\tau_0, \qquad \tau\colon H_1 \overset{\id}{\longrightarrow} H_\infty,\\
\tau\colon V_k {\longrightarrow} V_{-k}, \, z\mapsto \frac 1 {1-z}
\end{gather*}
Because $\tau$ preserves the preimages of the nodes of $\bar D_k$ it preserves the different $\Diff_{\bar D_k^\nu}(0)$ and thus by Theorem \ref{thm: triple} the triple $(\bar X, \bar D_k, \tau)$ determines uniquely a stable surface $X_k$.

We determine the singular points of the non-normal locus as described in Section~\ref{sect: non-normal locus}: for all $j=1, \dots, k$ the points $(\pm j,0) , (\pm j,1), (\pm j, \infty)$ are mapped to a single point $P_j\in X_k$ and every $P_j$ is a 6-multi-node of $D_k$. The non-normal locus $D_k$ has $k+2$ irreducible components: a smooth rational curve containing all $P_j$, which is the image of $H_0$, a nodal rational curve with a node at each $P_j$, which is the image of $H_1\cup H_\infty$, and for $j=1, \dots k$ rational curves $C_j=\pi(V_j\cup V_{-j})$ with a single $3$-multi-node at $P_j$. 

The only non-semi-smooth singularities of $X_k$ are degenerate cusps at the points $P_j$, where $X_k$ locally looks like the cone over a cycle of 6 independent lines. Thus $X_k$ is a Gorenstein stable surface.

We have $\chi(\ko_{D_k})=\chi(\ko_{D_k^\nu})-\chi(\nu_*\ko_{D_k^\nu}/\ko_{D_k})= k+2-5k=2-4k$. On the other hand it is easy to calculate $\chi(\ko_{\bar D_k}) = 3-4k$, so by Proposition~\ref{prop: invariants} the invariants of $X_k$ are $\chi(\ko_{X_k})=1+(2-4k)-(3-4k)=0$ and $K_{X_k}^2 = (K_{\bar X} + \bar D_k)^2 = 4k-4$. 

To prove that $|2K_X|$ has base points and $|3K_X|$ and $|4K_X$ are not very ample we analyse the restriction the  curves $C_j$. Its degree is
\[
\deg(\omega_{X_k}\restr{C_j}) = \frac{1}{2}(K_{\bar X_k} + \bar D_k)(V_j + V_{-j}) = 1.                                                                                                                                                                                                                                                                                                                                                                                                                                                                                                                                                                                                    \] 
Our claim now follows from the properties of line bundles of low degree on rational curves with a single 3-multi-node analysed in Example \ref{ex: 3-multi-nodal rational curve} below.
\end{exam}

\begin{exam}[A special curve]\label{ex: 3-multi-nodal rational curve}
 Let $B$ be a rational curve with a single $3$-multi-node, $\nu\colon B^\nu\to B$ its normalisation. Then $\chi(\ko_B) = \chi(\ko_{B^\nu}) -2 = -1$ and $B$ has arithmetic genus $p_a(B)=2$. 

For any line bundle  $\kl$ on $B$ the following properties hold.
\begin{enumerate}
 \item If $\deg \kl \geq 2$ then $H^1(B, \kl) =0$ and $h^0(B, \kl) = \deg\kl -1$.
\item If $\deg\kl=2$ then $h^0(B, \kl)=1$  and $\kl$ does not define a morphism.
\item If $\deg\kl=3$ or $\deg\kl=4$ then $|\kl|$ is base-point-free but not an embedding.
\item If $\deg\kl=4$ then $|\kl|$ defines a birational morphism, which is an embedding on the smooth locus.
\item If $\deg\kl\geq 5$ then $|\kl|$ defines an embedding.
\end{enumerate}
\begin{proof}
  By Serre duality $H^1(B, \kl) = \Hom_{\ko_B}(\kl, \omega_B)$. If $H^1(B, \kl) \neq 0$, then there is a non-zero $\lambda\colon \kl \to \omega_B$. As $\lambda$ is an isomorphism at the generic point and $\kl$ is torsion-free $\lambda$ is automatically injective and the cokernel is supported on a finite set of points. Thus 
 \[1 = \chi(\omega_B) \geq \chi(\kl) = \deg\kl+\chi(\ko_B) =\deg \kl-1\quad \iff \quad \deg \kl \leq 2.\]
As $\deg\kl \geq 2$ by the assumptions, we have $\deg \kl = 2$. Then $\lambda$ is an isomorphism. On the other hand, since $B$ has a 3-multi-node, $\omega_B$ is not locally free\,---\,a contradiction.  So there is no non-zero $\lambda$ and $H^1(B, \kl)=0$. This implies the formula for $h^0(B, \kl)$ by Riemann--Roch and we get \refenum{i}. The second item is an immediate consequence.

For \refenum{iii}, note that the embedding dimension of the 3-multi-node is 3 while  $\kl$ has at most 3 sections so we cannot have an embedding.

Note that, for $p$ a smooth point of $B$, part \refenum{i} applies to $\kl(-p)$ so $H^1(B, \kl(-p))=0$ and $p$ is not a base-point.
If $p$ is the 3-multi-node then the ideal sheaf $\ki_p$ of $p$ is $\nu_*\ko_{B^\nu}(-q_1-q_2-q_3)$, so 
\[ H^1(B, \kl\tensor\ki_p) = H^1(B^\nu, \nu^* \kl(-q_1-q_2-q_3))=0,\]
because $\nu^* \kl(-q_1-q_2-q_3))$ is a line bundle of non-negative degree on $\IP^1\isom B^\nu$. Thus also the 3-multi-nodal point is not a base-point.

If $p,p'$ are smooth points of $B$ then $H^1(B, \kl(-p-p'))=0$ if $\deg\kl\geq 4$ by \refenum{i} and $\kl$ separates smooth points and tangent vectors at smooth points. This proves \refenum{iv}. The last item follows from Theorem \ref{thm: curve embedding}.
\end{proof}
\end{exam}
\begin{rem}[Consequences of Example \ref{ex: 3-multi-nodal rational curve}]
 Assume that $X$ is a stable surface such that the non-normal locus $D$ contains a rational curve with a single 3-multi-node. If $\deg IK_X\restr B=1$ then
$\phi_{2I}$ is not a morphism (by Example \ref{ex: 3-multi-nodal rational curve}\,\refenum{ii}) and $\phi_{3I}$ and $\phi_{4I}$ are not an embedding (by Example \ref{ex: 3-multi-nodal rational curve}\,\refenum{iii}), because the respective restriction to $B$ has this property. In particular, this applies to Example \ref{exam: large K^2}.
\end{rem}

\appendix
\section{Curves on surfaces with slc singularities}
We fix some notation for this section: let $X$ be a surface with slc singularities and (possibly empty) non-normal locus $D$. Let $f\colon Y\rightarrow X$ be the minimal semi-resolution (Definition~\ref{defin: semi resolution}) with conductor divisor $D_Y$. In some instances when working near $p\in X$ we replace $X$ by a small affine or analytic neighbourhood of $p$.

Let $\nu\colon \bar Y \rightarrow Y$ and $\eta\colon \bar X\rightarrow X$ be the normalisations. We have a commutative diagram
\begin{equation}\label{equation: normalisation+resolution}
\begin{tikzcd}
 \bar Y \rar{\bar f} \dar{\eta} & \bar X\dar{\pi}\\ Y\rar{f} & X.
\end{tikzcd}
 \end{equation}
Denote by  $D_{\bar Y}\subset \bar Y$ and  $D_{\bar X}\subset \bar X$  the conductor divisors. Then
\begin{enumerate}
\item  $K_{\bar X} + \bar D$ is a $\IQ$-Cartier divisor and $K_{\bar Y} + D_{\bar Y}$ is a Cartier divisor.
 \item $D_{\bar Y}$ and $D_Y$ are smooth and $\eta\restr{D_{\bar Y}}\colon  D_{\bar Y} \rightarrow D_Y$ is a double cover.
 \item $D$ has at most \emph{$\mu$-multi-nodes} (see Definition~\ref{defin: multinode}), $\bar D$ has at most nodes and $\pi\restr{\bar D}\colon \bar D\rightarrow D$ is generically two to one. 
\end{enumerate} 

Now let  $B\subset X$ be a well-behaved curve. 
Our aim is to construct  a curve $\hat B_Y\subset Y$, the hat transform of $B$,  such that we control the difference of the arithmetic genera $p_a(B)$ and $p_a(\hat B_{\bar Y})$, where $\hat B_{\bar Y}\subset\bar Y$ is the strict transform of $\hat B_Y$. This will be achieved in Proposition~\ref{prop: genus hat transform}. The same idea has been used for surfaces with canonical singularities in \cite{CFHR}, but we have to work harder because our singularities are a lot worse.

Recall also that for a well-behaved divisor $A$ on $X$ and $C\subset X$ a curve, the sheaf $\ko_C(A)$ is the restriction $\ko_X(A)\restr C$ modulo torsion (Definition \ref{def: restriction divisorial sheaf}).

\subsection{Automatic adjunction lemma}
The following technical result will be used several times so we state it here for further reference.
\begin{lem}\label{lem: H^1 on curve}
 Let $C\subset X$ be a well-behaved curve and $A$ a well-behaved divisor on $X$. Then $H^1(C,\ko_C(A))\neq 0$ if and only if there is a non-empty subcurve $B\subset C$ with a generically onto $\lambda_B\colon \ko_B(A)\rightarrow \omega_B$. For such a subcurve $B$, we have $\chi(B,\ko_B(A))\leq \chi(B,\omega_B)$ and equality holds if and only if $\ko_B(A)\cong\omega_B$. We can choose $B$ to be connected.
\end{lem}

\begin{proof}
By Serre duality,  $H^1(C,\ko_C(A))\neq 0$ if and only if there is a non-zero homomorphism $\lambda\colon \ko_C(A) \rightarrow \omega_C$ in the dual space $\Hom(\ko_C(A),\omega_C)$. By automatic adjunction \cite[Lem.~2.4]{CFHR}, there is a subcurve $B$ of $C$ such that $\lambda$ restricts to a generically onto $\lambda_B\colon \ko_B(A)\rightarrow \omega_B$. 

On the other hand, if $B$ is a subcurve with a generically onto $\lambda_B\colon \ko_B(A)\rightarrow \omega_B$, then the composition 
\[
 \ko_C(A) \rightarrow \ko_B(A) \xrightarrow{\lambda_B} \omega_B \hookrightarrow \omega_C
\]
is the corresponding non-zero morphism from $\ko_C(A)$ to $\omega_C$. 

Since $\ko_B(A)$ is torsion free, the morphism $\lambda_B$ is injective. Being generically onto, $\lambda_B$ has a finite cokernel $\kq$. So we have the following short exact sequence
\[
 0\rightarrow \ko_B(A) \rightarrow \omega_B \rightarrow \kq\rightarrow 0,
\]
which yields 
\[
 \chi(B,\omega_B) = \chi(B,\ko_B(A))+ \mathrm{length}(\kq) \geq \chi(B,\ko_B(A)).
\]
This is an equality if and only if the length of $\kq$ is 0 which is in turn equivalent to $\lambda_B$ being an isomorphism.
\end{proof}

\subsection{Holomorphic Euler characteristics of well-behaved curves}

\begin{defin}\label{def: local correction}
 Let $F$ be a well-behaved curve on the semi-smooth surface $Y$ and  $F_{\bar Y}$ its strict transform in $\bar Y$. We denote by 
$I_t( F_{\bar Y}, D_{\bar Y})$  the intersection number of $ F_{\bar Y}$ and $D_{\bar Y}$ at a point $t\in\bar Y$.

For a point $q\in Y$ we will define the \emph{local genus correction} $n_q(F)$ of $F$ at $q$ 
and the \emph{local intersection difference} $d_q(F)$ of $F$ at $q$ such that the relation 
\[
 2 n_q(F) + d_q(F) = \sum_{t\in  \inverse \eta(q)} I_{t}( F_{\bar Y}, D_{\bar Y})
\]
holds. 
If $q$ is a normal crossing point with preimages  $t_1, t_2$ then 
\begin{gather*}
 n_q(F)=        \min \{I_{t_1}( F_{\bar Y}, D_{\bar Y}), I_{t_2}( F_{\bar Y}, D_{\bar Y})\},\\
d_q(F) = |I_{t_1}(F_{\bar Y}, D_{\bar Y})- I_{t_2}(F_{\bar Y}, D_{\bar Y})|.
\end{gather*}
If $q$ is a pinch point with preimage $t$ then 
\begin{gather*}
 n_q(F)=\lfloor {\textstyle\frac{1}{2}}I_{\inverse\eta(q)}( F_{\bar Y}, D_{\bar Y})\rfloor,\\
 d_q(F) = I_{t}( F_{\bar Y}, D_{\bar Y}) -2\lfloor {\textstyle\frac{1}{2}}I_{t}( F_{\bar Y}, D_{\bar Y})\rfloor.
\end{gather*}
At a smooth point $q\in Y$ we set $n_q(F)=d_q(F)=0$.
\end{defin}
\begin{rem}\label{rem: linearity and convexity of n_q}
Let $F, G$ be two well-behaved curves on a semi-smooth surface $Y$. Elementary arithmetics with minimum and round down  show that, for $q\in\sing Y$,
\[
 -\min\{d_q(F),d_q(G)\}\leq  n_q(F) + n_q(G) - n_q(F+G)\leq 0
\]
and the  inequality on the right hand side is an equality if one of $F$ and $G$ is Cartier at $q$.

We call $q$ a bad point  with respect to $F$ and $G$ if $n_q(F) + n_q(G) - n_q(F+G)<0$; we have $d_q(F),d_q(G)\geq 1$ for bad points $q$.
\end{rem}

We now use these locally defined corrections to prove global identities for the holomorphic Euler-characteristic of well-behaved curves on semi-smooth surfaces. 
\begin{prop}\label{prop: semi-smooth RR}
Let $F$ and $G$ be well-behaved curves on $Y$ and $F_{\bar Y}\subset \bar Y$ (resp.\ $G_{\bar Y}$) the strict transforms of $F$ (resp.\ $G$). Then
\begin{enumerate}
\item $\chi(Y,\ko_Y(-F)) = \chi(Y,\ko_Y) - \chi(F_{\bar Y}, \ko_{F_{\bar Y}}) + \sum_{q\in\sing Y} n_q(F)$.
\item $\chi(F,\ko_F) = \chi(F_{\bar Y}, \ko_{F_{\bar Y}}) - \sum_{q\in\sing Y} n_q(F)$;
\item $\chi(G,\ko_G(-F)) = \chi(G,\ko_G) - FG + \sum_{q\in\sing Y}n_q(F) + n_q(G) - n_q(F+G)$.
\end{enumerate}
\end{prop}
\begin{proof}
Recall  that the map
\[
  \eta\restr{D_{\bar Y}}\colon D_{\bar Y} \rightarrow D_Y
\]
is a double cover between smooth curves; the branch locus of $\eta\restr{D_{\bar Y}}$ consists exactly of the pinch points of $Y$. Thus
\[
 \eta_*\ko_{D_{\bar Y}} = \ko_{D_Y} \oplus \kl^{-1}
\]
where $\kl$ is a line bundle on $D_Y$ with $\kl^{\otimes2}=\ko_{D_Y}(\sum_{p \text{ pinch point}} p)$. 

There is a commutative diagram of sheaves with exact rows and columns:
\begin{equation}\label{diag: semi-smooth RR}
\begin{tikzcd}
  {}& 0\dar & 0\dar & 0 \dar  \\
 0 \rar & \ko_Y(-F)\rar\dar & \ko_Y\rar\dar & \ko_F\rar\dar & 0 \\
 0 \rar & \eta_*\ko_{\bar Y}(-F_{\bar Y})\rar\dar & \eta_*\ko_{\bar Y} \rar\dar & \eta_*\ko_{ F_{\bar Y}} \rar\dar  & 0\\
 0 \rar & \km \rar\dar & \inverse \kl \rar\dar & \kr \rar\dar & 0\\
 & 0 & 0 & 0
\end{tikzcd}
\end{equation}
where $\km$ is cokernel of the natural morphism $\ko_Y(-F)\rightarrow\eta_* \ko_{\bar Y}(-F_{\bar Y})$. Here, because of the Snake Lemma, the last row is exact at $\km$. Using the additivity of the Euler characteristic, we have
\begin{align}
 \chi(Y,\ko_Y(-F)) &= \chi(Y,\ko_Y) - \chi(Y,\ko_F) \notag\\
                 &= \chi(Y,\ko_Y)  - \chi(Y,\eta_*\ko_{F_{\bar Y}}) +\chi(Y,\kr)\\
                  & = \chi(Y,\ko_Y)  - \chi(F_{\bar Y},\ko_{F_{\bar Y}}) +\chi(Y,\kr)\notag\hspace{0.5cm}\text{(since $\eta$ is finite).}
\end{align}

Note that $\kr$ is finite with support in $F\cap {D_Y}$. For \refenum{i} it suffices to prove the following claim.

\textbf{Claim.} For a point  $q\in \sing Y$ we have $\dim_\IC \kr_q = n_q(F)$, where $n_q$ is the local genus correction defined above.
\begin{proof}[Proof of the claim.]
We can calculate $\kr$ analytically locally around $q$.

If $q$ is a double normal crossing point of $Y$, then analytically locally $Y$ is $(xy=0)\subset\IC^3_{x,y,z}$ with $q=(0,0,0)$ and $D_Y=(x=y=0)$. The normalisation is  $\bar Y = \IC^2_{x,z_1} \sqcup\IC^2_{y,z_2}$ and the preimages $t_1, t_2$ are the origins in the components. The cokernel of inclusion of the coordinate rings $\IC[Y]\hookrightarrow \IC[\bar Y]$ is isomorphic to $\IC[D_Y] = \IC[z]$. If $F_{\bar Y}$ is defined by $f(x,z_1)$ in one irreducible component $\IC^2_{x,z_1}$ and by $g(y,z_2)$ in the other irreducible component $ \IC^2_{y,z_2}$, then the image of its defining ideal $\ki_{F_{\bar Y}} = (f(x,z_1),g(y,z_2))$ in the localised ring $\IC[D_Y]_q = \IC[z]_{(z)}$ is the ideal of $\IC[z]_{(z)}$ generated by $f(0,z)$ and $g(0,z)$. Note that the orders of $f(0,z)$ and $g(0,z)$ in $z$ are just the intersection numbers $I_{t_i}(F_{\bar Y}, D_{\bar Y})$, $i=1,2$. Therefore we have
\begin{align*}
  \dim_\IC  \kr_q & = \dim_\IC \IC[z]_{(z)}/(f(0,z),g(0,z)) \\
                  & =  \min\{I_{t_1}(F_{\bar Y}, D_{\bar Y}),I_{t_2}(F_{\bar Y}, D_{\bar Y})\}                   
\end{align*}

If $q$ is a pinch point of $Y$, then analytically locally $Y$ is $(x^2 - y^2 z =0)\subset \IC^3$ with $p=(0,0,0)$ and $D_Y=(x=y=0)$. The normalisation $\bar Y$ of $Y$ is $\IC^2_{u,y}$ with normalisation map
\[
 \begin{matrix}
  \bar Y & \rightarrow & Y  \\
    (u,y) & \mapsto & (uy,y,u^2).
 \end{matrix}
\]
The preimage $t$ of $q$ is the origin in $\bar Y = \IC^2_{u,y}$ and the conductor divisor $D_{\bar Y}$ is defined by $y=0$ in $\bar Y$. The cokernel of the inclusion of coordinate rings $\IC[Y] \hookrightarrow \IC[\bar Y]$ is naturally isomorphic to $u\IC[D_Y] = u\IC[z]$. (Note that in $\IC[\bar Y]$ we have $z=u^2$.) Now suppose $F_{\bar Y}$ is defined by $f(u,y)=0$ locally around $q=(0,0)$. Then $\ki_{F_{\bar Y}}$ is the ideal of $\IC[u,y]$ generated by $f(u,y)$. Note that the order of $f(u,0)$ is just the intersection number $I_t(F_{\bar Y}, D_{\bar Y})$.  Then the image of $\ki_{F_{\bar Y}}$ in the localised module $u\IC[z]_{(z)}$ is the submodule generated by $uz^{k}$, where $\dim_\IC \kr_q=k = \lfloor \frac{1}{2}I_t(F_{\bar Y}, D_{\bar Y})\rfloor$. This concludes the proof.
\end{proof}

The first row of \eqref{diag: semi-smooth RR} gives $\chi(F,\ko_F) = \chi(Y,\ko_Y) - \chi(Y,\ko_Y(-F))$,
so \refenum{ii} is implied by \refenum{i}.

The short exact sequence $ 0\rightarrow \ko_Y(-F-G) \rightarrow \ko_Y(-F) \rightarrow \ko_G(-F) \rightarrow 0$ from Lemma \ref{lem: div seq}   gives $\chi(G, \ko_G(-F))= \chi(Y,\ko_Y(-F)) - \chi(Y,\ko_Y(-F-G))$. Applying \refenum{i} to both terms on the right hand side and then substituting \refenum{ii} gives \refenum{iii}.
\end{proof}

\subsection{Resolution graphs and semi-rationality}\label{sect: semi-rat}

We will now recall some more of the classification of slc singularities. The resolution graphs of  log-canonical surface singularities are well known (e.g. \cite[Ch.~4]{Kollar-Mori}) so we concentrate on the non-normal case; our sources are  \cite[17]{kollar12} and \cite[Sect.~4]{ksb88}.

Over a non-normal point $p\in X$ we can write $\bar Y$ analytically locally $\bar Y = \cup \bar Y_\alpha$ as the union of local irreducible components. On each component the $f$-exceptional divisors together with the components of the double locus give rise to an (extended) dual graph: every  $f$-exceptional component, which are all rational because we are over a non-normal point of $X$, gives a vertex which is either marked with ``$\circ$'' or with the negative self-intersection; we add a ``$\bullet$'' for every component of the  conductor divisor $D_{\bar Y}$ and connect two vertices if the corresponding curves intersect.

The edges  connecting the resolution graph to the boundary components are marked with the coefficient of the different $\Diff_{D_{\bar X}}(0)$ at the corresponding point of the conductor divisor on $\bar X$. 

The following three cases can occur:
\[\tag{C1}\label{C1}
\bullet   \ \stackrel{1-\frac1{\delta}}{-\!\!\!-\!\!\!-\!\!\!-} 
\ c_1  \ -\ \cdots  \ - \ c_n\qquad (c_i\geq 2)  
\]
where $-\delta$ is the determinant of the intersection form of the exceptional divisors.

\[\tag{C2}\label{C2}
\bullet   \ \stackrel{1}{-}\ c_1  \ -\ \cdots  \ - \ c_n  \ \stackrel{1}{-}
\ \bullet \qquad (c_i\geq1)
\]
and if some $c_j=1$, then $n=1$, because we consider the minimal semi-resolution.
 
\[\tag{Dh}\label{Dh}
\begin{array}{cccccccc}
 &&&&&& 2&\\
&&&&& \diagup &\\
\hspace{2.5cm}\bullet   \ \stackrel{1}{-}\  c_1 & - & \cdots  & - & c_n &&&\hspace{1.5cm} (n\geq 2,\, c_i\geq 2).\\
&&&&& \diagdown &\\
 &&&&&& 2&
\end{array}
\]
According to the type of extended dual graph associated to the curves in an irreducible component $\bar Y_\alpha$ we say $Y_\alpha:=\eta(\bar Y_\alpha)$ is of type \eqref{C1} (resp.\ \eqref{C2}, \eqref{Dh}).

The whole extended dual graph of an slc singularity is obtained by attaching graphs of the types \eqref{C1}, \eqref{C2}, and \eqref{Dh} along the boundary components, with the restriction that the differents should match (see Theorem \ref{thm: triple}). If the exceptional divisor intersects the conductor on $Y$ in  a pinch point, then there is a boundary components in the resolution graph which is not glued to any other component (compare \cite[Prop.~4.27]{ksb88} for more details). In total we see that the exceptional divisors form a tree of rational curves unless we glue a number of components of type \eqref{C2} in a circle, that is, the singularity is a degenerate cusp.

The following is an important property of a singularity.
\begin{defin}[{\cite[Def.~4.14]{ksb88}, \cite[Def.~4.1.1]{vS87}}]
If $R^1f_*\ko_Y =0$ then we say $X$ has semi-rational singularities. 
\end{defin}
Morally all results valid for rational singularities hold also in the semi-rational case. 

The following makes the connection to slc singularities.
\begin{lem}
Let $p\in X$ be an slc surface singularity and $f\colon Y\rightarrow X$ the minimal semi-resolution.

The point  $p$ is not semi-rational if and only if $p$ is simple elliptic, a cusp, or a degenerate cusp.
In this case $\dim_\IC (R^1f_*\ko_Y)_p=1$. (This number is sometimes called the geometric genus of a singularity.)
\end{lem}
\begin{proof}
In the normal case this is well known, see for example \cite[Lem.~9.3]{kawamata88}. 
In lack of an appropriate reference we sketch a proof in the non-normal case.

The statement that $\dim_\IC (R^1f_*\ko_Y)_p=1$ for a degenerate cusp is contained in \cite[Thm.~4.3.6]{vS87}. So it remains to prove that in all other cases the singularity is semi-rational. 

Suppose first $p\in X$ is a non-normal Gorenstein point  but not a degenerate cusp. By \cite[Thm. 4.21]{ksb88} $p$ is a normal crossing or a pinch point and thus semi-rational.
So it remains to show that a non-Gorenstein slc singularity $p\in X$ is semi-rational.

Let $\tilde p\in\tilde X$ be the canonical index one cover of $p\in X$. Then $\tilde X \rightarrow X$ is a $G$-covering branched only at the single point $p$ (\cite[Thm.~4.24]{ksb88}) for some finite group $G$.

If $\tilde p\in\tilde X$ is semi-canonical, i.e., a normal crossing or pinch point, then we can argue as in \cite[Thm.~1]{kovacs00}, which works for semi-rationality as well.

Otherwise $\tilde p\in\tilde X$ is a degenerate cusp.  Let $\tilde g\colon\tilde W\rightarrow \tilde X$ the minimal semi-resolution. Then the $G$-action lifts to $\tilde W$  and we have a commutative diagram
\[
\begin{tikzcd}
        {\tilde W} \dar{\tilde g}\rar{\lambda} & W:=\tilde W/G \arrow[shift right = .7cm]{d}{g}\\
             \tilde X \rar{\phi} & X = \tilde X/G
\end{tikzcd}
\]
Arguing again as in \cite[Thm.~1]{kovacs00}, we see that $W$ has semi-rational singularities.

Denote by $\varphi_*^G$ (resp.\ $\lambda_*^G$) be the composite functor of pushforward and taking the $G$-invariant part. Then $\varphi_*^G$ (resp.\ $\lambda_*^G$) is an exact functor from the category of $G$-equivariant $\ko_{\tilde X}$-modules (resp.\ $\ko_{\tilde W}$-modules) to the category of $\ko_X$-modules (resp.\ $\ko_W$-modules). Then $g_*\circ \lambda_*^G =  \varphi_*^G\circ g$,  and by the Grothendieck spectral sequence we get
\[
 R^1 g_* \ko_W \isom R^1(g_*\circ \lambda_*^G)(\ko_{\tilde W}) \isom  R^1(\varphi_*^G\circ \tilde g)(\ko_{\tilde W}) \isom  \varphi_*^G R^1\tilde g_*(\ko_{\tilde W})\isom  (R^1\tilde g_*(\ko_{\tilde W}))^G.
\]
With the same argument as in \cite[Proof of Thm.~9.6, p.~143]{kawamata88} one proves that $G$ acts effectively on $R^1\tilde g_*(\ko_{\tilde W}) \isom \IC$ thus  $R^1\tilde g_*(\ko_{\tilde W}))^G=0$ and $p\in X$ is semi-rational also in this case. This concludes the proof.
\end{proof}
An alternative approach to this result is to compute the fundamental cycle on a stable improvement in the sense of \cite{vS87} and then use \cite[Thm.~4.1.3]{vS87}.

\begin{rem}\label{rem: chi}
Let $p\in X$ be an slc point and $Y$  the semi-resolution of a sufficiently small neighbourhood of $p$. Then for every effective divisor $E$ supported on the exceptional divisors  we have a surjection
\[ (R^1f_*\ko_Y)_p\isom H^1(Y, \ko_Y) \onto H^1(E, \ko_E).\]
Thus if $p$ is semi-rational then $h^1(E,\ko_E)=0$ and $\chi(E,\ko_E) = h^0(E,\ko_E)\geq 1$; if $p$ is not semi-rational then $h^1(E,\ko_E)\leq 1$ and $\chi(E,\ko_E) \geq  h^0(E,\ko_E)-1\geq 0$.

More precisely, in the non-semi-rational case equality can only occur if the support of $E$ is the full exceptional locus since otherwise $E$ is supported on the exceptional divisor of a semi-rational singularity.
\end{rem}

\subsection{Semi-numerical cycle} \label{sect: numerical cycle}

\begin{defin}\label{defin: semi-numerical cycle}
 Let $p\in X$ be a  non-semi-smooth point and $E_i$ the exceptional divisors over $p$. The semi-numerical cycle $Z$ over $p$ is a minimal Weil divisor $Z=\sum_i a_i E_i$ of $Y$ such that
\begin{enumerate}
 \item $a_i\in \IZ$ for any $i$;
 \item $(Z_{\bar Y} + D_{\bar Y})E_{i,\bar Y} \leq 0$ for any $i$.
\end{enumerate}
\end{defin}

\begin{rem}\label{rem: semi-numerical cycle}
 If $X$ is normal, then $D_{\bar Y}$ is empty and Definition~\ref{defin: semi-numerical cycle} coincides with the usual definition of  the numerical cycle (\cite[Sect.~4.5]{Reid97}); the existence and the uniqueness of a semi-numerical cycle is proved in the same fashion as for the normal singularities, using the negative definiteness of the intersection form on the exceptional curves. The semi-numerical cycle turns out to carry the cohomology $R^1f_*\ko_Y$ (cf.\ Remark~\ref{rem: Z carries cohomology}).  

See \cite[3.4]{vS87} for a  discussion of the notion of fundamental cycle for a more general class of  non-normal surfaces singularities.
\end{rem}
In the case where $p\in X$ is normal, the numerical cycle is nicely elaborated in \cite[Section 4]{Reid97} (see also \cite[Thm.~4.7]{Kollar-Mori}), so we concentrate on the non-normal case. 

\begin{rem}[Semi-numerical cycle on non-normal slc singularities]\label{rem: semi-numerical}
Let $p\in X$ be a non-normal slc point. Locally analytically around the preimage of $p$ we decompose the resolution $Y$ into irreducible components of the type presented in Section \ref{sect: semi-rat} and correspondingly  the semi-numerical cycle $Z=\bigcup_\alpha Z_\alpha$ where $Z_\alpha \subset Y_\alpha$. Since the intersection form is defined via the normalisation the divisors $Z_\alpha$ are uniquely determined by the configuration of exceptional curves and boundary components on $Y_\alpha$.

Computing in each of the different cases we see that $Z_\alpha$ is the reduced sum of exceptional divisors except in the following cases:
\begin{enumerate}
 \item The component $Y_\alpha$ is of type \eqref{C2} with extended dual graph
\[\bullet   \ - \ 1  \ -\ \bullet \]
 and $Z_\alpha=2E$, where $E$ is the exceptional curve.
\item The component is of type \eqref{Dh} with dual graph
\[\begin{array}{cccccccc}
 &&&&&& 2&\\
&&&&& \diagup &\\
\bullet   \ {-}\  2 & - & \cdots  & - & 2 &&&\\
&&&&& \diagdown &\\
 &&&&&& 2&
\end{array}\]
and, denoting by $E'$ and $E''$ the two exceptional curves on the right of the fork and by $E_j$ the exceptional curves in the chain to the left of the fork,  the restriction of the semi-numerical cycle is
$Z_\alpha=E' + E'' + 2\sum_{j=1}^{n}E_{j}.$
\end{enumerate}
In particular,  $Z$ has multiplicity at most $2$ at each irreducible exceptional curve. 
\end{rem}

The next result shows the importance of the semi-numerical cycle for the computation of higher pushforward sheaves.
\begin{lem}\label{lem: R^1 = H^1}
Let $C$ and $F$ be well-behaved curves on $X$ and $Y$ respectively such that $f_*F = C $ as Weil divisors. Suppose moreover $F E\leq 0$ for any effective exceptional divisor $E$ over $p$. Then $R^1f_*\ko_Y(-F)_p\isom H^1(Z, \ko_Z(-F))$ where $Z$ is the semi-numerical cycle over $p$.
\end{lem}
\begin{rem}\label{rem: Z carries cohomology}
 Applying the above to an empty curve we see that   $H^1(Z,\ko_Z) = (R^1 f_*\ko_Y)_p$ where $p\in X$ and $Z$ is the semi-numerical cycle over $p$.
\end{rem}
\begin{proof} 
If $p$ is semi-smooth then the map $f$ is finite in a neighbourhood of $p$ and both sides are $0$. So assume $p$ to be a non-semi-smooth point of $X$. Let  $E$ be any effective  divisor supported on the exceptional locus over $p$.
The restriction sequence from Lemma~\ref{lem: div seq} yields an  an exact sequence in cohomology: 
\begin{equation}\label{eq: thickening numerical cycle}
 H^1(E, \ko_E(-Z - F)) \rightarrow H^1(Z + E, \ko_{Z+E}(-F) ) \rightarrow H^1(Z,\ko_{Z}(-F) ) \rightarrow 0.
\end{equation}

\textbf{Claim.}  $H^1(E, \ko_E(-Z - F)) =0.$
 \begin{proof}[Proof of the claim]

Suppose on the contrary that  $H^1(E, \ko_E(-Z - F)) \neq0$. By Lemma~\ref{lem: H^1 on curve}, there is a subcurve $E'\subset E$ such that
\begin{equation}\label{eq: deg of Z+B_Y on E'}
 \chi(E',\ko_{E'}(-Z - F)) \leq \chi(E',\omega_{E'}) = -\chi(E',\ko_{E'}). 
\end{equation}
Recall that  by Remark \ref{rem: chi}
\begin{equation}\label{eq: chi}
 \begin{split}
&\chi(E', \ko_{E'}) \geq 0 \text{ and }\\ &\chi(E', \ko_{E'}) \geq 1\text{ unless $p$ is not semi-rational and $\supp E' = \inverse f (p)$.}  
 \end{split}
\end{equation}

Applying Proposition~\ref{prop: semi-smooth RR}\refenum{iii} to $\ko_{E'}(-Z - F)$ equation \eqref{eq: deg of Z+B_Y on E'} becomes 
\begin{equation}\label{eq: positive negative}
 2\chi(E',\ko_{E'}) + \sum_{q\in\sing Y}n_q(E') + n_q(Z+F) - n_q(E'+ Z+F) \leq (Z + F)E'
\end{equation}

We treat the case  $p\in X$ a normal point first. Then $Y$ is smooth and the above equation becomes
\[
 2\chi(E',\ko_{E'})\leq (Z + F)E'\leq 0.
\]
If $p\in X$ is rational or $\supp(E')\neq \inverse f(p)$, then $2\chi(E',\ko_{E'})\geq 2$ by \eqref{eq: chi}\,---\,a contradiction. If $p\in X$ is not rational and $\supp(E') = \inverse f(p)$ then  $E'=Z+E''$ for some effective $E''$. Hence  $(Z + F)E'\leq ZE'=Z^2 + Z E''\leq Z^2<0$  while  $2\chi(E',\ko_{E'})\geq 0$ which again contradicts \eqref{eq: chi}. 

Now we can assume $p\in X$ is non-normal. Let $E'_{\bar Y}\subset\bar Y$ be the strict transform of $E'$. Then 
\begin{equation}\label{eq: estimate error term}
  n_q(Z+F) - n_q(E'+ Z+F)\geq -\sum_{t\in \eta^{-1}(q)}I_t(E'_{\bar Y},D_{\bar Y}) \geq ZE'
\end{equation}
where the last inequality is because of the definition of the semi-numerical cycle. Combining \eqref{eq: positive negative} and \eqref{eq: estimate error term}, we have
\[
 2\chi(E',\ko_{E'}) + \sum_{q\in\sing Y}n_q(E') \leq FE'\leq 0.
\]
where the last inequality is by our assumption on $F$.

 If $p\in X$ is non-semi-rational with $\supp(E')=\inverse f(p)$, then $\sum_{q\in\sing Y}n_q(E')>0$ and $\chi(E',\ko_{E'})\geq 0$ by \eqref{eq: chi}\,---\,contradiction. Otherwise $\chi(E',\ko_{E'})\geq 1$ and again we get a contradiction.

Thus a subcurve $E'$ as in \eqref{eq: deg of Z+B_Y on E'} cannot exist and $H^1(E, \ko_E(-Z-F))=0$ as claimed.
\end{proof}
Therefore the sequence \eqref{eq: thickening numerical cycle} yields $H^1(Z,\ko_{Z}(-F) )\isom H^1(Z + E, \ko_{Z+E}(-F) )$ for every effective exceptional divisor $E$ over $p$. Moreover, the surjection $\ko_{Y}(-F)\restr{Z+E} \onto \ko_{Z+E}(-F)$ induces an isomorphism  
\[H^1(Z + E, \ko_{Z+E}(-F) )\isom  H^1(Z + E, \ko_{Y}(-F)\restr{Z+E} )\]
 because the kernel is supported on points. By the theorem on formal functions (\cite[Thm.~III.11.1]{Hartshorne}) we have
$R^1f_*(\ko_Y(-F))_p = H^1 (Z,\ko_{Z}(-F))$ as claimed.
\end{proof}

We now  give some lower bounds on the Euler characteristic of subcurves of semi-numerical cycles.
\begin{lem}\label{lem: subcurves of Z}
Let $Z$ be the semi-numerical cycle over $p\in X$ and $E\subset Z$ a connected subcurve. Then 
\[
 2\chi(E,\ko_E)\geq \sum_{q\in\sing Y} d_q(E)
\]
with equality if and only if $E$ satisfies one of the following
\begin{enumerate}
 \item $p\in X$ is simple elliptic singularity or a cusp and $E=Z$.
 \item $p\in X$ is non-normal, $E$ is reduced and every connected component of $\bar E$ is a chain of smooth rational curves  intersecting $D_{\bar Y}$ in two  points. 
\end{enumerate}
\end{lem}
\begin{proof}
If $p\in X$ is normal, then $d_q(E) = 0$ for all $q$ and the assertions follow from Remark \ref{rem: chi}. So in the following we assume $p\in X$ is non-normal and that $E$ is connected.

Denote by $\bar E$ the strict transform of $E$ in $\bar Y$. 
Let $\bar E^{(\alpha)}$ ($1\leq \alpha \leq n$) be the connected components of $\bar E$ and $E^{(\alpha)}:=\eta_* \bar E^{(\alpha)}$ the pushforward as Weil divisors. We arrange the labels  in such a way that $E^{(\alpha)}\cap E^{(\alpha+1)}\neq\emptyset$ for $1\leq \alpha\leq n-1$.  By Proposition~\ref{prop: semi-smooth RR}, we have 
\begin{equation}\label{eq: chi E}
 \chi(E,\ko_E)  = \chi(\bar E,\ko_{\bar E}) - \sum_{q\in\sing Y} n_q(E)= \sum_{\alpha=1}^n\chi(\bar E^{(\alpha)},\ko_{\bar E^{(\alpha)}}) - \sum_{q\in\sing Y} n_q(E), 
\end{equation}
and similarly
\begin{equation}\label{eq: chi E_red}
 \chi(E_\red,\ko_{E_\red})   = \sum_{\alpha=1}^n\chi(\bar E_\red^{(\alpha)},\ko_{\bar E_\red^{(\alpha)}}) - \sum_{q\in\sing Y} n_q(E_\red).
\end{equation}

First we look at the reduction $E_\red$ of $E$, which we assumed to be connected. If $p\in X$ is either semi-rational or non-semi-rational but $\supp(E)\neq \inverse f(p)$, then $E_\red$ is a reduced tree of rational curves; we have $\chi(E_\red,\ko_{E_\red}) = 1$ and  since a local intersection difference can only occur at the end points $\sum_{q\in\sing Y} d_q(E_\red)\leq 2$. 

 If $p\in X$ is a non-semi-rational non-normal point then it is a degenerate cusp, and if $\supp(E)=\inverse f(p)$ the divisor $E_\red$ is a cycle of rational curves so that $\chi(E_\red,\ko_{E_\red}) = 0$ and  $\sum_{q\in\sing Y} d_q(E_\red) = 0$.
In both cases we have 
\begin{equation}\label{eq: E_red}
2\chi(E_\red,\ko_{E_\red})-\sum_{q\in\sing Y} d_q(E_\red) \geq 0
\end{equation}
with equality if and only if $E_\red$ is as described in \refenum{ii}.

In general, $E$ is obtained from $E_\red$ by adding some irreducible components of $E_\red$. More precisely, let $F_1,\dots,F_k$ be the irreducible components of $E -E_\red$ so that we can write $E=E_\red + \sum_{1\leq j\leq k} F_j$. We order in such a way that  $F_{1},\dots, F_{k'}$ have non-empty intersection with $D_Y$ while $F_{k'+1},\dots,F_{k}$ do not intersect $D_Y$.

Using the computation of semi-numerical cycles from Remark \ref{rem: semi-numerical} we distinguish three possible cases for $F_j\subset Y_{(\alpha_j)}$.
\begin{enumerate}
 \item[a)] $F_j$ is a $(-1)$-curve. Then $E^{(\alpha_j)}=2F_j$ because the only possibility is type \eqref{C2} of length 1. Then by the  adjunction formula on $\bar Y$, we have $\chi(\bar E^{(\alpha_j)}, \ko_{\bar E^{(\alpha_j)}}) = 3$. Also  $\sum_{q\in\sing Y}\sum_{t\in\inverse\eta(q)} I_t(\bar F_j,D_{\bar Y}) = 2$ for such a curve $F_j$.
 \item[b)] $F_j$  lies in an irreducible component  $Y_{\alpha_j}\subset Y$ of type \eqref{Dh} and intersects $D_Y$. Write $\bar E^{(\alpha_j)}= \bar E^{(\alpha_j)}_1 + \bar F_j$. Computing Euler characteristics for the structure sequence of $\bar E^{(\alpha_j)}_1\subset \bar E^{(\alpha_j)}$ in the explicit situation of Remark \ref{rem: semi-numerical} one obtains that
$\chi(\bar E^{(\alpha_j)}, \ko_{\bar E^{(\alpha_j)}})\geq 2$. For such an $F_j$, we have $\sum_{q\in\sing Y}\sum_{t\in\inverse\eta(q)} I_t(\bar F_j,D_{\bar Y}) = 1$.
\item[c)] $F_j$ lies in an irreducible component $Y_\alpha\subset Y$ of type \eqref{Dh} and but there is no non-reduced irreducible component of $E_\alpha$ intersecting the conductor. Thus  $\bar E^{(\alpha_j)}$ is supported on the exceptional divisor of a rational surface singularity and $\chi(\bar E^{(\alpha_j)}, \ko_{\bar E^{(\alpha_j)}})\geq 1$  by \cite[Prop.~4.12]{Reid97}.
\end{enumerate}

Let $r_1:=\#\{j\,|\,1\leq j\leq k', F_j \text{ is a $(-1)$-curve}\}$ and $r_2:=k'-r_1$. Then since by classification  $\chi(\bar E^{(\alpha)}_\red, \ko_{\bar E^{(\alpha)}_\red})=1$ for every connected component of $\bar E$ in total we get
\begin{equation}\label{eq: comparison chi bar E and E_red}
 2\sum_{\alpha=1}^n\chi(\bar E^{(\alpha)},\ko_{\bar E^{(\alpha)}}) -  2\sum_{\alpha=1}^n\chi(\bar E_\red^{(\alpha)},\ko_{\bar E_\red^{(\alpha)}}) \geq 2(2r_1 + r_2).
\end{equation}
On the other hand, we have
\begin{align*}
  &\sum_{q\in\sing Y} 2n_q(E_\red)+d_q(E_\red) - \sum_{q\in\sing Y} 2n_q(E)+d_q(E)\\
=&\bar E_\red D_{\bar Y}-\bar E D_{\bar Y} \qquad (\text{by Definition \ref{def: local correction}})\\
=& -\sum_{j=1}^{k'} \bar F_jD_{\bar Y}\\
  = &- (2r_1 + r_2)
\end{align*}
Adding these equations and using \eqref{eq: chi E}, \eqref{eq: chi E_red} and \eqref{eq: E_red}, we have
\[
 2\chi(E, \ko_E) - \sum_{q\in\sing Y} d_q(E) \geq  2\chi(E_\red, \ko_{E_\red}) - \sum_{q\in\sing Y} d_q(E) + 2r_1 + r_2 \geq 2r_1 + r_2 \geq 0.
\]
If equality holds then $r_1=r_2=0$ and $ 2\chi(E_\red, \ko_{E_\red}) =\sum_{q\in\sing Y} d_q(E)$ so  $E_\red$ is as described in \refenum{ii}. As a consequence, no irreducible component of  $E_\red$ is contained  in a component of type \eqref{Dh} and there cannot be non-reduced irreducible components not intersecting the conductor. Thus $E=E_\red$ and $E$ is as in \refenum{ii}. 

On the other hand it is easy to see that equality holds if $E$ is as in \refenum{ii}.
\end{proof}

\begin{lem}\label{lem: comparison Rf_*}
Let $F\subset Y$ be a well-behaved curve such that $FE_i\leq 0$ for any exceptional curve $E_i$ over $p\in X$. Then we have 
\begin{align*}
 \dim_{\IC}R^1f_*\ko_Y(-F)_p & \leq \dim_{\IC}(R^1f_*\ko_Y)_p +\frac{1}{2}\#\left\{q\in\inverse f(p)\,|\,d_q(F)>0\right\},
\end{align*}
where the function $d_q$ is as in Definition~\ref{def: local correction}.
\end{lem}
\begin{proof}
If $p\in X$ is semi-smooth then the map $f$ is an isomorphism in a neighbourhood of $p$ and $\dim_{\IC}R^1f_*\ko_Y(-F)_p=0$. So we may assume that $p$ is a non-semi-smooth singularity of $X$.  Let $Z\subset Y$ the semi-numerical cycle  over $p$. By Lemma~\ref{lem: R^1 = H^1} we have 
\[
 R^1f_*(\ko_Y(-F))_p = H^1 (Z,\ko_{Z}(-F)).
\]
and it remains to estimate the dimension of the right hand side.

Suppose $H^1 (Z,\ko_{Z}(-F)) \neq 0$. Then, by Lemma~\ref{lem: H^1 on curve}, there is a connected subcurve $E\subset Z$ such that $ \chi(E,\ko_{E}(-F))\leq \chi(E,\omega_E)=-\chi(E, \ko_E)$ with equality if and only if $\ko_{E}(-F)\cong\omega_E$. Combining with Proposition~\ref{prop: semi-smooth RR}\refenum{iii} yields
\begin{equation}\label{eq: comparison of degrees}
 2\chi(E,\ko_E) + \sum_{q\in\sing Y} n_q(E) + n_q(F) - n_q(E + F) \leq F E, 
\end{equation}
with equality if and only if $\ko_{E}(-F)\cong \omega_E$. By Remark~\ref{rem: linearity and convexity of n_q}
\[
 n_q(E) + n_q(F) - n_q(E + F) \geq -\min\{d_q(E),d_q(F)\}\geq  -d_q(E)
\]
and hence
\begin{align*}
 2\chi(E,\ko_E) + \sum_{q\in\sing Y} n_q(E) + n_q(F) - n_q(E + F) \geq 2\chi(E,\ko_E) - \sum_{q\in\sing Y} d_q(E) \geq 0
\end{align*}
where the last inequality comes from Lemma~\ref{lem: subcurves of Z}.
Since $FE\leq 0 $ by assumption we have equality  in \eqref{eq: comparison of degrees} and 
 $\ko_{E}(-F)\cong \omega_E$. This implies $FE =0$, 
\[2\chi(E,\ko_E) = \sum_{q\in\sing Y} d_q(E),\] 
and $n_q(E) + n_q(F) - n_q(E + F) = -\min\{d_q(E),d_q(F)\}=  -d_q(E)$ for all $q\in \sing Y$.
In particular, 
\begin{equation}\label{eq: EF}
d_q(F) \geq d_q(E). 
\end{equation}

\paragraph{\textbf{Case 1:} $p\in X$ is normal.}
By Lemma~\ref{lem: subcurves of Z}, $p\in X$ is either a simple elliptic singularity or a cusp  and $E=Z$. In particular
\[
 h^1(Z,\ko_Z(-F)) =h^1(Z,\omega_Z) = h^1(Z,\ko_Z) = 1,
\]
and we have equality in the claim of the Lemma.

\paragraph{\textbf{Case 2:} $p\in X$ is non-normal.} 
Let $\ke$ be set of connected subcurves $E$ of $Z$ such that there is a generically onto homomorphism $\lambda_E\colon \ko_E(-F)\rightarrow \omega_E$.

\textbf{Claim:} If $E \neq E' \in \ke$ then $E$ and $E'$ have disjoint support.
\begin{proof}
Interpreting the morphisms $\lambda\colon \ko_E(-F) \to \omega_E$ and $\lambda'\colon \ko_{E'}(-F) \to \omega_{E'}$ as elements in $H^1(Z, \ko_Z(-F))$ a general linear combination will give a generically onto morphism supported on $E\cup E'$. 
Thus our claim follow if we can show that for a curve $E\in \ke$ no  connected proper  subcurve can be contained in $\ke$. 

By Lemma~\ref{lem: subcurves of Z}, $E$ lies completely in the union of irreducible components of type \eqref{C2}, so  $E$ is a reduced nodal curve of arithmetic genus 0 or 1. By the above $\ko_E(-F)\isom \omega_E$. For every connected proper subcurve $E'\subset E$ we have $\deg \ko_E(-F)\restr {E'} = \deg \omega_E\restr {E'}\geq -1 > -2 = \deg \omega_{E'}$ and thus there is no generically onto morphism from $\ko_{E'}(-F)$ to $\omega_{E'}$.
\end{proof}

Since different curves in $\ke$ are disjoint, we have
\begin{equation}\label{eq: H^1}
 H^1(Z,\ko_Z(-F))=\bigoplus_{E\in\ke}H^1(E,\ko_E(-F))=\bigoplus_{E\in\ke}H^1(E,\omega_E)\isom \IC^{\#\ke}.
\end{equation}

If $p_a(E)=1$ for some $E\in \ke$ then $p$ is a degenerate cusp and $E$ is the reduced preimage of $p$. Thus $\ke=\{E\}$ and 
$ H^1(Z,\ko_Z(-F))=H^1(E,\omega_E)$ and the claimed inequality holds.

Otherwise by Lemma~\ref{lem: subcurves of Z}, every  $E\in\ke$ is a chain of rational curves such that the  end(s) of the chain intersect the conductor in two (different) points $q_1$ and $q_2$; at these intersection points $q_i$ we have $1=d_{q_i}(E)\leq d_{q_i}(F)$  by \eqref{eq: EF}. Since  every two different curves in $\ke$ are disjoint the following inequality holds
\[
\#\ke \leq \frac{1}{2}\#\{q\in \inverse f(p)\,|\,d_q(F)>0\}.
\]
Together with equation \eqref{eq: H^1} and Lemma \ref{lem: R^1 = H^1} this completes the proof of the lemma.
\end{proof}

\subsection{A relative duality}
As a preparation for the relative duality result we need the following lemma.
\begin{lem}\label{lem: Ext^1}
Let $C$ be a well-behaved curve and $A$ a well-behaved divisor on $X$. Then 
 $\Ext^1_{\ko_X}(\ko_C,\omega_X(A)) \cong \Ext^1_{\ko_X}(\ko_C(-A),\omega_X).$
\end{lem}
 \begin{proof}
We look at the structure sequence
\[  0\rightarrow \ko_X(-C) \rightarrow \ko_X \rightarrow \ko_C \rightarrow 0 
\]
and, applying $\shom_{\ko_X}(\cdot,\omega_X(A))$, obtain an exact sequence
\[
 0\rightarrow \omega_X(A) \rightarrow \shom_{\ko_X}(\ko_X(-C),\omega_X(A)) \rightarrow \shext^1_{\ko_X}(\ko_C,\omega_X( A)) \rightarrow 0.                                                             
\]                                                             
Since $\omega_X(A)$ is $S_2$, $\shom_{\ko_X}(\ko_X(-C),\omega_X(A))$ is also $S_2$ by \cite[Lem.~5.1.1]{abr-hass11}. Therefore we have  $\shom_{\ko_X}(\ko_X(-C),\omega_X(A))\cong \omega_X(C+A)$, since the two coincide outside a finite set of points and both are $S_2$. So there is a short exact sequence
 \begin{equation}\label{eq: exact sequence ext_C1}
 0\rightarrow \omega_X(A) \rightarrow  \omega_X(C+A) \rightarrow \shext^1_{\ko_X}(\ko_C,\omega_X( A)) \rightarrow 0.                                                             
\end{equation}

Applying $\shom_{\ko_X}(\cdot,\omega_X)$ to the restriction sequence 
\[ 0\rightarrow \ko_X(-C-A) \rightarrow \ko_X(-A) \rightarrow \ko_C(-A) \rightarrow 0\]
 from Lemma~\ref{lem: div seq}, we get
\[
\begin{split}
 0\rightarrow \shom_{\ko_X}(\ko_X(-A),\omega_X) \rightarrow \shom_{\ko_X}(\ko_X(-C-A),\omega_X)\\
 \rightarrow \shext^1_{\ko_X}(\ko_C(-A),\omega_X) \rightarrow  \shext^1_{\ko_X}(\ko_X(-A),\omega_X( A)).  
\end{split}                                                           
\]  
As before, by the $S_2$ property of the relevant sheaves, there are isomorphisms
\begin{gather*}
  \shom_{\ko_X}(\ko_X(-A),\omega_X)\cong\omega_X(A),\\
\shom_{\ko_X}(\ko_X(-C-A),\omega_X)\cong \omega_X(C+A).
\end{gather*}
Also, we have  $\shext^1_{\ko_X}(\ko_X(-A),\omega_X( A)) = 0$ by  Lemma~\ref{lem: ext reflexive}, which leaves us with a short exact sequence
\begin{equation}\label{eq: exact sequence ext_C2}
 0\rightarrow \omega_X(A) \rightarrow  \omega_X(C+A) \rightarrow \shext^1_{\ko_X}(\ko_C(-A),\omega_X) \rightarrow 0.                                                             
\end{equation}

Comparing \eqref{eq: exact sequence ext_C1} and \eqref{eq: exact sequence ext_C2} gives $\shext^1_{\ko_X}(\ko_C,\omega_X( A)) \cong \shext^1_{\ko_X}(\ko_C(-A),\omega_X)$.
Now, $\shom_{\ko_X}(\ko_C,\omega_X( A))$ and $\shom_{\ko_X}(\ko_C(-A),\omega_X)$ both being zero, the claim  follows from the local-to-global-Ext-spectral-sequence.  
\end{proof}

\begin{prop}\label{lem: rel. duality}
Let $C$ and $F$ be well-behaved curves on $X$ and $Y$ respectively such that $f_*F = C $ as Weil divisors. Then, for any $p\in X$, the vector spaces $R^1f_*\ko_Y(-F)_p$ and $\left(\omega_X(C)/f_*\omega_Y(F)\right)_p$ are dual to each other.
\end{prop}
\begin{proof}
Let $E$ be the reduced exceptional divisor over $p$. 
As in \cite[Lem.~3.3.3]{kollar85} we have, using Lemma~\ref{lem: Ext^1}, 
\begin{align*}
\left(\omega_X(C)/f_*\omega_Y(F)\right)_p &= H^1_E(\omega_Y( F)) \\&= \varinjlim_n \Ext^1_{\ko_Y}(\ko_{nE},\omega_Y(F))\cong\varinjlim_n\Ext^1_{\ko_Y}(\ko_{nE}(-F),\omega_Y)
\end{align*}
 The surjection $\ko_Y(-F)\restr {nE}\onto \ko_{nE}(- F)$ has torsion kernel and thus by Serre duality 
\[\Ext^1_{\ko_Y}(\ko_{nE}(- F),\omega_Y) = H^1(Y,\ko_{nE}(- F))^\vee=H^1(Y,\ko_{Y}(- F)\restr{nE})^\vee.\]
Combining these equations we have by the theorem of formal functions (\cite[Thm.~III.11.1]{Hartshorne}) 
\[
\left(\omega_X(C)/f_*\omega_Y(F)\right)_p \cong \varprojlim_n H^1(Y,\ko_{Y}(- F)\restr{nE})^\vee = R^1 f_*\ko_Y(- F)^\vee_p
\]
which concludes the proof.
\end{proof}

\subsection{The hat transform}\label{section: hattransform}
Since the intersection form is negative definite on the exceptional divisors of $f\colon Y\rightarrow X$, we have (cf.~\cite[Lemma~3.41]{Kollar-Mori})
\begin{lem}\label{lem: eff div}
 Let $\tilde B$ and $\tilde C$ be two well-behaved divisors on $Y$. Assume that $f_*\tilde B= f_* \tilde C$ and $\tilde C E\leq \tilde B E$ for any exceptional divisor $E$ of $f$. Then $\tilde B\leq \tilde C$.
\end{lem}
\begin{propdef}\label{defin: hattransform}
Let $B\subset X$ be a well-behaved  curve. Then there exists a unique  well-behaved curve $\hat B_Y\subset Y$ which is is minimal with respect to the properties $f_*\hat B_Y= B$ and  for all exceptional divisors $E$ of $f$ 
\[\hat B_Y E\leq 0.\]
We call $\hat B_Y$ the hat transform of $B$ with respect to $f$. 
\end{propdef}
\begin{proof}
We can take a well-behaved very ample Cartier divisor $H$ of $X$ that contains $B$. Then $f^*H - (\inverse f)_*(H-B)$ contains the strict transform of $B$ and has non-positive intersection with any exceptional divisor $E_i$ for any $i$. The existence follows. 

Suppose $\hat B_1$ and $\hat B_2$ are two hat transforms of $B$ under $f$. Then we have inequalities $\hat B_i E \leq 0 \leq B_Y E$ ($i=1,2$), where $B_Y$ is the strict transform of $B$ on $Y$. So both $\hat B_1$ and $\hat B_2$ are effective divisors by Lemma~\ref{lem: eff div}. Write $\hat B_1=\hat B_3 + A_1, \hat B_2=\hat B_3 + A_2$, where $A_1$ and $A_2$ are two well-behaved effective divisors with no common irreducible components. Let $ E$ be a reduced and irreducible exceptional divisor of $ f$. If $E\subset A_1$ then $E\nsubseteq A_2$, and $\hat B_3 E = (\hat B_2 - A_2)E\leq 0$; if $E\subset A_2$ then $E\nsubseteq A_1$, and $\hat B_3 E = (\hat B_1 - A_1)E\leq 0$. By the minimality of a hat transform we have $\hat B_1=\hat B_2 = \hat B_3$. The uniqueness is proved.
\end{proof}
We start to gather some properties of the hat transform.
\begin{lem}\label{lem: easy properties of hattransform}
Let $\hat B_Y \subset Y$ be the hat transform of $B$. Then the following holds.
\begin{enumerate} 
 \item  $\hat B_Y-B_Y$ contains only exceptional curves of $f\colon  Y\rightarrow X$.
 \item  Let $B_Y^*:=f^* B_Y = B_Y + \Gamma^*$ be the numerical pullback of $B_Y$, so that $B_Y^* E_i =0$ for any exceptional curve $E_i$ of $ f\colon Y \rightarrow  X$ . Then $\hat B_Y\geq B_Y^*$. 
 \item If $C\subset X$ is an effective Cartier divisor such that $B\leq C$ then $\hat B\leq f^* C$.
\end{enumerate}
In particular, if $B$ is Cartier then $\hat B_Y = B_Y^*$. 
\end{lem}
\begin{proof}
Recall that $\hat B_Y$ is well-behaved. 

For \refenum{i}, if $\hat B_Y-B$ contains some curve  $A$ that is not exceptional then $(\hat B_Y - A)E_i\leq 0$ for any exceptional $E_i$, contradicting the minimality of $\hat B_Y$. 

For \refenum{ii}, note that $B_Y^*E=0$ for any exceptional divisor $E$. So $\hat B_Y E\leq B_Y^* E$ for any exceptional curve $E$. Since $\hat B_Y - B_Y^*$ is supported only on exceptional divisors, we have  $\hat B_Y\geq B_Y^*$ by Lemma~\ref{lem: eff div}.

For \refenum{iii}, note that $f^*C$ is an integral divisor, since $C$ is Cartier. Moreover $f^*C E =0$ for any exceptional curve $E$, and the strict transform $B_Y$ of $B$ is contained in $f^*C$. By the minimality of $\hat B_Y$, the inequality $\hat B_Y\leq f^* C$ follows.
\end{proof}

\begin{rem}
 Note that since we used normalisation to define intersection numbers both $\hat B_Y$ and $B^*_Y$ can behave unexpectedly: they  might not contain all exceptional curves mapping to $B$. See also Remark \ref{rem: intersection properties}.
\end{rem}

\begin{prop}\label{prop: GRvanishing} 
In the situation above we have
 \begin{enumerate}
  \item $R^1f_*\omega_Y(\hat B_Y) =0$;
   \item $ R^1 f_*\omega_{\hat B_Y}  = 0$;
 \item $\chi(\omega_{\hat B_Y}) = \chi(f_* \omega_{\hat B_Y})$.
\end{enumerate}
\end{prop}
\begin{proof}
For \refenum{i}  we look at the diagram (\ref{equation: normalisation+resolution}) and consider the exact sequence
\[
 0\rightarrow \eta_*\omega_{\bar Y}(\hat B_{\bar Y}) \rightarrow \omega_Y(\hat B_Y) \rightarrow \kq_Y \rightarrow 0
\]
where $\hat B_{\bar Y}$ is the strict transform of $\hat B_Y$ on $\bar Y$ and $\kq_Y$ is supported on $D_Y$. Applying $f_*$ we have
\begin{equation}\label{equation: higherimage}
 R^1 f_*\eta_* \omega_{\bar Y}(\hat B_{\bar Y}) \rightarrow R^1f_*\omega_Y(\hat B_Y) \rightarrow R^1f_*\kq_Y.
\end{equation}
Since $f\restr{D_Y}$ is finite, $R^1f_*\kq_Y=0$. On the other hand, using the Leray spectral sequence and finiteness of $\eta$ and $\pi$, we have
\begin{align*}
 R^1f_*\eta_*\omega_{\bar Y}(\hat B_{\bar Y}) & = R^1(f\eta)_* \omega_{\bar Y}(\hat B_{\bar Y}) \\
                                             & = R^1(\pi\bar f)_*\omega_{\bar Y}(\hat B_{\bar Y})\\
                                             & = \pi_* R^1\bar f_*\omega_{\bar Y}(\hat B_{\bar Y}).
\end{align*}
The argument for \cite[Claim 4.3 (iv)]{CFHR} gives $R^1\bar f_*\omega_{\bar Y}(\hat B_{\bar Y}) =0$ and hence $R^1f_*\eta_*\omega_{\bar Y}(\hat B_{\bar Y})=0$. Now the vanishing of $ R^1f_*\omega_Y(\hat B_Y)$ follows from the exact sequence (\ref{equation: higherimage}).

For \refenum{ii}, we apply $\shom_{\ko_Y}(\cdot,\omega_Y)$ to the short exact sequence 
\[
 0\rightarrow \ko_Y(-\hat B_Y) \rightarrow \ko_Y \rightarrow \ko_{\hat B_Y} \rightarrow 0,
\]
and get
\[0\rightarrow \omega_Y \rightarrow \omega_Y(\hat B_Y) \rightarrow \omega_{\hat B_Y} \rightarrow 0.\]
Applying $f_*$ to the above sequence and using \refenum{i} gives $R^1 f_* \omega_{\hat B_Y} \cong R^1 f_* \omega_Y(\hat B_Y)=0$.

The last item  is a direct consequence of \refenum{ii} and the Leray spectral sequence.
\end{proof}

We now prove the main result of the section, an estimate for the change in arithmetic genus of the strict transform of the hat transform.
\begin{prop}\label{prop: genus hat transform}
Let $B$ be a well-behaved curve on $X$ and $\hat B_{\bar Y}\subset\bar Y$ the strict transform of the hat transform of $B$. Then 
\[p_a(B)\leq p_a(\hat B_{\bar Y}) + \frac{\hat B_{\bar Y} D_{\bar Y}}{2}.\] 
\end{prop}
\begin{proof}
The short exact sequence $0\to \omega_X\to \omega_X(B) \to \omega_B \to 0$ fits into the diagram 
\[
 \begin{tikzcd}
{}   & 0\dar & 0\dar&\kk_B\arrow[hookrightarrow]{d}\\
0\rar & f_*\omega_Y\rar\dar&f_*\omega_Y(\hat B_Y)\rar\dar& f_* \omega_{\hat B_Y}\rar\dar{\exists}& 0 \\
0\rar & \omega_X\rar\dar& \omega_X(B)\rar\dar& \omega_B\rar\dar& 0 \\
 & \omega_X/f_*\omega_Y\rar\dar& \omega(X)(B)/f_*\omega_Y(\hat B_Y)\rar\dar& \kq_B \rar\dar& 0 \\
        & 0  & 0 & 0.
 \end{tikzcd}
\]
The first row is exact by Proposition~\ref{prop: GRvanishing} and the rest of the diagram is exact by defining $\kk_B$ (resp.\ $\kq_B$) to be the kernel (resp.\ cokernel) of $f_*\omega_{\hat B_Y} \rightarrow \omega_B$. The Snake Lemma
together with the duality from Lemma \ref{lem: rel. duality} gives \[
  \dim_{\IC}\kk_B -  \dim_{\IC}\kq_B = h^0(R^1f_*\ko_Y) -  h^0(R^1f_*\ko_Y(-\hat B_Y)).
\]
Now we have 
\begin{align*}
  p_a(\hat B_Y)& = 1 - \chi(\ko_{\hat B_Y}) \\
              &= 1 + \chi(\omega_{\hat B_Y})  \\
&= 1 + \chi(f_*\omega_{\hat B_Y})  \qquad\text{(by Prop.\ \ref{prop: GRvanishing}\refenum{iii})} \\
              &= 1 + \chi(\omega_B) +  \dim_{\IC}\kk_B -  \dim_{\IC}\kq_B \\
              &= 1 +  \chi(\omega_B) +  h^0(R^1f_*\ko_Y) -  h^0(R^1f_*\ko_Y(-\hat B_Y))\\
              &\geq p_a(B) - \frac{1}{2}\#\left\{q\in\hat B_Y\cap D_Y\,|\,d_q(\hat B_Y)>0\right\}\qquad\text{(by Lem.\ \ref{lem: comparison Rf_*})}\\
              &\geq p_a(B) - \frac{1}{2}\sum_{q\in \sing Y} d_q(\hat B_Y)
\end{align*}
Thus, using  Proposition~\ref{prop: semi-smooth RR}\refenum{ii} for $\hat B_Y$, we get 
 \begin{align*}
 p_a(B) &\leq  p_a(\hat B_{Y})+ \frac{1}{2}\sum_{q\in \sing Y} d_q(\hat B_Y)\\
&=  p_a(\hat B_{\bar Y})+ \sum_{q\in \sing Y} n_q(\hat B_Y)+ \frac{1}{2}\sum_{q\in \sing Y} d_q(\hat B_Y)\\
& = p_a(\hat B_{\bar Y}) + \frac{\hat B_{\bar Y} D_{\bar Y}}{2}.
\end{align*}
where the last equality is by Definition \ref{def: local correction}. This concludes the proof.
\end{proof}

% \bibliographystyle{../halpha}
% \bibliography{../srollens}
\end{document}